\documentclass[12pt,reqno]{article}
\usepackage[usenames]{color}
\usepackage{amssymb}
\usepackage{amsmath}
\usepackage{amsthm}
\usepackage{amsfonts}
\usepackage{amscd}
\usepackage{graphicx}

\usepackage[colorlinks=true,
linkcolor=webgreen,
filecolor=webbrown,
citecolor=webgreen]{hyperref}

\definecolor{webgreen}{rgb}{0,.5,0}
\definecolor{webbrown}{rgb}{.6,0,0}

\usepackage{color}
\usepackage{fullpage}
\usepackage{float}

\usepackage{graphics}
\usepackage{latexsym}
\usepackage{epsf}
\usepackage{breakurl}

\newtheorem{thm}{Theorem}[section]
\newtheorem{prop}[thm]{Proposition}

\newtheorem{cor}[thm]{Corollary}
\newtheorem{lemma}[thm]{Lemma}
\theoremstyle{definition}
\newtheorem{defn}[thm]{Definition}
\newtheorem{rmk}[thm]{Remark}
\newtheorem{ex}[thm]{Example}
\newtheorem{question}[thm]{Question}
\DeclareMathOperator{\Ap}{Ap}

\begin{document}

\begin{center}
\vskip 1cm{\LARGE\bf On Free Numerical Semigroups and the Construction of Minimal Telescopic Sequences} \vskip 1cm
\large Caleb M. Shor\\
Department of Mathematics\\
Western New England University\\
Springfield, MA 01119\\
USA\\
\href{mailto:cshor@wne.edu}{\tt cshor@wne.edu}\\
\end{center}
\vskip .2 in
\begin{abstract}
A free numerical semigroup is a submonoid of the non-negative integers with finite complement that is additively generated by the terms in a telescopic sequence with gcd 1. However, such a sequence need not be minimal, which is to say that some proper subsequence may generate the same numerical semigroup, and that subsequence need not be telescopic. In this paper, we will see that for a telescopic sequence with any gcd, there is a minimal telescopic sequence which generates the same submonoid. In particular, given a free numerical semigroup we can construct a telescopic generating sequence which is minimal. In the process, we will examine some operations on and constructions of telescopic sequences in general.
\end{abstract}

\section{Introduction}\label{sec:new-intro}
Let $\mathbb{N}$ denote the set of positive integers and let $\mathbb{N}_0=\mathbb{N}\cup\{0\}$. For a set $A\subseteq\mathbb{N}_0$, let $\langle A\rangle$ denote the set of all (finite) $\mathbb{N}_0$-linear combinations of elements of $A$. Let $S$ be a submonoid of $\mathbb{N}_0$. All submonoids of $\mathbb{N}_0$ are finitely generated, so $S=\langle A\rangle$ for some finite $A=\{a_1,\dots,a_k\}\subset\mathbb{N}_0$. We also write this as $S=\langle a_1,\dots,a_k\rangle$. We say a set $A$ is \emph{minimal} if $\langle A'\rangle\ne\langle A\rangle$ for all $A'\subsetneq A$. It is well known that any set has a unique minimal subset that generates the same submonoid, and hence there exists a bijection between submonoids of $\mathbb{N}_0$ and minimal subsets of $\mathbb{N}_0$. (See Rosales and Garc\'ia-S\'anchez \cite[Cor.\ 2.8]{RosalesGarciaSanchez09} for details.)

In this paper, we are interested in generating sets that are ordered in some way. We therefore consider (finite) sequences $G=(g_1,\dots,g_k)$ for some $k\in\mathbb{N}$ with $g_i\in\mathbb{N}_0$, which we denote by $G\in\mathbb{N}_0^k$.\footnote{By taking $G\in\mathbb{N}_0^k$, we allow for the possibility that $g_i=g_j$ for $i\ne j$. We can also have $g_i=0$ for some $i$. Clearly, if we remove any repeats or 0s, the resulting subsequence will still generate the same submonoid. However, our results hold in the more general case of allowing repeats and/or zeros, so we choose to allow them to give slightly more general results.} We similarly define the submonoid generated by $G$ as $\langle G\rangle = \langle g_1,\dots,g_k\rangle$. As with sets, a sequence $G$ is \emph{minimal} if $\langle G'\rangle\ne\langle G\rangle$ for all proper subsequences $G'$ of $G$. Any sequence then has a unique minimal subsequence that generates the same submonoid. Observe that any permutation of a minimal sequence is also necessarily minimal. 

For $G\in\mathbb{N}_0^k$ a sequence with $g_1+g_2>0$, let $G_i=(g_1,\dots,g_i)$ and $d_i=\gcd(G_i)$ for $1\le i\le k$. Let $c(G)=(c_2,\dots,c_k)$ where $c_j=d_{j-1}/d_j$ for $2\le j\le k$.\footnote{Observe that $c_j$ is a well-defined integer when $d_j\ne 0$. We have $d_j\ne 0$ for all $j$ precisely when $g_1+g_2>0$.} If $c_jg_j\in\langle G_{j-1}\rangle$ for all $j\ge 2$, then we say $G$ is a \emph{telescopic} (or \emph{smooth}) {sequence}. Some authors include the additional properties where $\gcd(G)=1$ and/or $G$ is an increasing sequence. In this paper, we require neither.

The reader can confirm that all sequences $G\in\mathbb{N}_0^k$ with $g_1+g_2>0$ for $k=1$ and $k=2$ are telescopic. There are longer sequences which are not telescopic, though for which a permutation is. One example is $G=(4, 5, 6)$, which is not telescopic, though $G'=(4,6,5)$ is telescopic. And finally, there are sequences for which no permutation is telescopic, such as $G=(3,4,5)$.

Telescopic sequences arise naturally in the world of numerical semigroups. They generate so-called \emph{free numerical semigroups}, which have some nice properties that we will see along with references for further reading in Section \ref{sec:background}. Unfortunately, the usage of the word ``free'' here does not coincide with the categorical idea of free objects.

The motivation for this paper is how the definitions of ``telescopic'' and ``minimal'' interact. Consider the following example.

\begin{ex}\label{ex:question-example}
Let $S=\langle G\rangle$ for $G=(660, 550, 352, 50, 201)\in\mathbb{N}_0^5$. We see that $c(G)=(6, 5, 11, 2)$ and that $c_jg_j\in\langle G_{j-1}\rangle$ for $2\le j\le 5$, so $G$ is a telescopic sequence. Since $\gcd(G)=1$, $S$ is a free numerical semigroup.  
However, $G$ is not minimal because $550=0\cdot660+0\cdot352+11\cdot 50+0\cdot201$. We eliminate 550 from $G$ to obtain the proper subsequence $G'=(660, 352, 50, 201)$ with $\langle G'\rangle=\langle G\rangle=S$. We can see that $G'$ is minimal. However, $G'$ is not telescopic.
\end{ex}
This leads to the following question.
\begin{question}\label{q:1}
Given a telescopic sequence $G$, does there exist a telescopic sequence $G'$ which is minimal and has $\langle G'\rangle=\langle G\rangle$?
\end{question}
We can ask this more generally.
\begin{question}\label{q:2}
Suppose $G$ and $H$ are sequences such that $\langle G\rangle=\langle H\rangle$. If $G$ is telescopic, must some permutation of $H$ be telescopic?
\end{question}

As Example \ref{ex:question-example} illustrates, for $G$ telescopic and $G'$ its unique minimal subsequence, $G'$ need not be telescopic. In this particular case, the permutation $(660, 50, 352, 201)$ of $G'$ is telescopic. We will show that this always happens -- i.e., that the answer to Question \ref{q:1} is ``yes,'' as is the answer to the more general Question \ref{q:2}. In the context of numerical semigroups, this means that any free numerical semigroup is generated by a telescopic minimal sequence, and we will give a procedure to compute it.

\subsection{Organization}
This paper is organized as follows. In Section \ref{sec:background}, we will give some background material on free numerical semigroups. Following that, given a sequence $(c_2,\dots,c_k)$ and some $d\in\mathbb{N}$, we provide an explicit method (Remark \ref{rmk:how-to-construct-telescopic}) in Section \ref{sec:explicit-telescopic-construction} to produce any telescopic sequence $G=(g_1, \dots, g_k)$ such that $c(G)=(c_2, \dots, c_k)$ and $\gcd(G)=d$. In Section \ref{sec:sequence-operations}, we describe two functions which map telescopic sequences to telescopic sequences. As we will see, any function which maps telescopic sequences to telescopic sequences (with the same gcd) must be a composition of these functions (with certain parameters). 

Then, in Section \ref{sec:reduction-telescopic-minimal} we present a construction which takes a telescopic sequence as input and outputs a minimal telescopic sequence that generates the same submonoid, answering Question \ref{q:1} affirmatively. The main result is Theorem \ref{thm:minimal-telescopic-sequence}. As a corollary, we find that every free numerical semigroup is generated by a minimal telescopic sequence. We also answer the more general Question \ref{q:2} affirmatively.

Finally, in Section \ref{sec:explicit-minimal-telescopic-construction} we combine two results (Remark \ref{rmk:how-to-construct-telescopic} and Proposition \ref{prop:two-cases}) to give an explicit method to construct any minimal telescopic sequence $G$ with $c(G)=(c_2, \dots, c_k)$. As an application, we have Corollary \ref{cor:increasing-telescopic-minimal} which says that any non-decreasing telescopic sequence $G=(g_1, \dots, g_k)$ with $c_j>1$ for all $j$ is necessarily minimal.

\section{Background}\label{sec:background}
Free numerical semigroups, which are generated by telescopic sequences $G$ with $\gcd(G)=1$, have been studied in various contexts by Brauer and Shockley \cite{Brauer1962}, Herzog \cite{Herzog1970}, Bertin and Carbonne \cite{BertinCarbonne1977}, R\"odseth \cite{Rodseth1978}, Kirfel and Pellikaan \cite{KirfelPellikaan1995}, Rosales and Garc\'ia-S\'anchez \cite{Rosales1999,RosalesGarciaSanchez09}, Leher \cite{Leher2007}, Ayano \cite{Ayano2014}, Robles-P\'erez and Rosales \cite{Robles-PerezRosales2018}, Gassert and Shor \cite{GassertShor18}, and others. We will highlight some of their properties.
 
\subsection{Numerical semigroups}
We begin with numerical semigroups. For a comprehensive reference on the subject, see the work of Rosales and Garc\'ia-S\'anchez \cite{RosalesGarciaSanchez09}.

A \emph{numerical semigroup} $S$ is a submonoid of $\mathbb{N}_0$ with finite complement.  It is well-known that every numerical semigroup $S$ is given by $S=\langle A\rangle$ for some finite $A\subset\mathbb{N}_0$ with $\gcd(A)=1$. Since we are interested in generating sequences, we equivalently have that every numerical semigroup $S$ is given by $S=\langle G\rangle$ for some $k\in\mathbb{N}$ and $G\in\mathbb{N}_0^k$ with $\gcd(G)=1$. Elements of the complement of $S$ are known as \emph{gaps} of $S$, and we denote the set of gaps by $H(S)$. The \emph{genus} of $S$, denoted $g(S)$, is the number of gaps of $S$. The \emph{Frobenius element} of $S$, denoted $F(S)$, is the largest integer not in $S$, which is $\max(H(S))$ when $S\ne \mathbb{N}_0$. The \emph{embedding dimension} of $S$, denoted $e(S)$, is the cardinality of the unique minimal generating set of $S$ (which is the cardinality of any minimal generating sequence of $S$).

Given a numerical semigroup $S$, it can be difficult to compute $g(S)$, $F(S)$, and other properties of the set of gaps. Curtis \cite{Curtis1990} showed that there cannot be a polynomial formula to compute the Frobenius number of $S$ as a function of the generating elements of $S$ when $e(S)>2$ . Ram\'irez Alfons\'in \cite{RamirezAlfonsin96} proved that the problem of computing the Frobenius number of $S$, when $e(S)>2$, is NP-hard. For more on the Frobenius problem, see Ram\'irez Alfons\'in's book \cite{RamirezAlfonsin05}.

A very helpful tool in understanding the gaps of a numerical semigroup is called the Ap\'ery set \cite{Apery1946}. For any nonzero $t\in S$, the \emph{Ap\'ery set of $S$ relative to $t$} is 
\begin{equation}\label{eqn:apery-defn}
\Ap(S;t) = \{s\in S : s-t\not\in S \}.
\end{equation}  
Equivalently, $\Ap(S;t)$ is the set of elements in $S$ which are minimal in their congruence class modulo $t$.

If we know $\Ap(S;t)$, then we immediately know the genus of $S$ (from Selmer \cite{Selmer1977}) and the Frobenius number of $S$ (from Brauer and Shockley \cite{Brauer1962}).

\begin{thm}[{\cite[Lem.\ 3]{Brauer1962}}, {\cite[Eqn.\ 2.3]{Selmer1977}}]\label{thm:brauer-and-selmer}
For $S$ a numerical semigroup and any nonzero $t\in S$,
\begin{equation}\label{eqn:frob-formula}
F(S)=\max(\Ap(S;t))-t,
\end{equation}
and
\begin{equation}\label{eqn:genus-formula}
g(S)=\frac{1-t}{2}+\frac{1}{t}\sum\limits_{n\in\Ap(S;t)}n.
\end{equation}
\end{thm}

We can also deduce other properties of the gaps of $S$ with the following identity and an appropriately chosen function $f$.

\begin{thm}[{\cite[Thm.\ 2.3]{GassertShor18}}]\label{thm:tuenter-apery}
For $S$ a numerical semigroup with set of gaps $H(S)$, any nonzero $t\in S$, and any function $f$ defined on $\mathbb{N}_0$,
\begin{equation}\label{eqn:tuenter-apery}
\sum\limits_{n\in H(S)}[f(n+t)-f(n)]=\sum\limits_{n\in\Ap(S;t)}f(n) - \sum\limits_{n=0}^{t-1}f(n).
\end{equation}
\end{thm}

For instance, with $f(n)=n$ in Equation \eqref{eqn:tuenter-apery}, we obtain the genus formula in Equation \eqref{eqn:genus-formula}. Tuenter \cite{Tuenter06} presents additional applications in the case where $S=\langle a,b\rangle$.

\subsection{Free numerical semigroups}
A numerical semigroup $S$ is a \emph{free numerical semigroup} if $S=\langle G\rangle$ for a telescopic sequence $G$. The following result from Bertin and Carbonne \cite[Sect.\ 2]{BertinCarbonne1977} motivates our interest in free numerical semigroups. We present it for any value of $\gcd(G)$, whereas it usually appears in the case of $\gcd(G)=1$.
\begin{thm}
Let $G=(g_1, \dots, g_k)\in\mathbb{N}_0^k$ be telescopic, let $c(G)=(c_2, \dots, c_k)$, and let $d=\gcd(G)$. For any $n\in d\mathbb{Z}$, there is a unique representation 
\begin{equation}
n=\sum\limits_{i=1}^k n_i g_i,
\end{equation}
with integers $n_1, \dots, n_k$ where $0\le n_j<c_j$ for $j = 2, \dots, k$. Furthermore, for such an $n\in d\mathbb{Z}$, we have $n\in\langle G\rangle$ if and only if $n_1 \ge 0$.
\end{thm}

This theorem is a generalization of the result that, given relatively prime positive integers $a$ and $b$, any integer $n$ can be uniquely written as $n=n_1a+n_2b$ with $n_1,n_2\in\mathbb{Z}$ and $0\le n_2<a$.

We then obtain an explicit description of the Ap\'ery set of a free numerical semigroup relative to the first generating element.

\begin{cor}[{\cite[Prop.\ 3.6]{GassertShor18}}]\label{cor:explicit-apery}
Suppose $S$ is a free numerical semigroup, so $S=\langle G\rangle$ for $G=(g_1, \dots, g_k)\in\mathbb{N}_0^k$ a telescopic sequence with $\gcd(G)=1$. For $c(G)=(c_2, \dots, c_k)$, we have 
\begin{equation}
\Ap(S;g_1)=\left\{\sum\limits_{j=2}^k n_j g_j : 0\le n_j < c_j \right\}.
\end{equation}
\end{cor}

Using Corollary \ref{cor:explicit-apery} with Theorem \ref{thm:brauer-and-selmer} we obtain the formulas for the Frobenius number and genus of a free numerical semigroup. The formula for $F(S)$ in this case was known to Brauer \cite{Brauer1942}.

\begin{cor}
For $S$ a free numerical semigroup generated by $G=(g_1, \dots, g_k)\in\mathbb{N}_0^k$, a telescopic sequence with $\gcd(G)=1$ and $c(G)=(c_2, \dots, c_k)$,
\begin{equation}\label{eqn:frobenius-free}
F(S) = -g_1+\sum\limits_{j=2}^k (c_j-1)g_j,
\end{equation}
and
\begin{equation}\label{eqn:genus-free}
g(S) = \frac{1}{2}\left(1-g_1+\sum\limits_{j=2}^k (c_j-1)g_j\right).
\end{equation}
\end{cor}
Observe that for $S$ a free numerical semigroup we have $F(S)=2g(S)-1$. Numerical semigroups with this property are \emph{symmetric}. These semigroups are equivalently characterized as those for which the map $\phi:S\to\mathbb{Z}\setminus S$ given by $\phi(s)=F(S)-s$ is a bijection.

We can combine Corollary \ref{cor:explicit-apery} with Theorem \ref{thm:tuenter-apery} to obtain an explicit identity for the gaps of a free numerical semigroup.

\begin{cor}[{\cite[Cor.\ 3.7]{GassertShor18}}]\label{cor:tuenter-telescopic-id}
Suppose $S=\langle G\rangle$ for $G=(g_1, \dots, g_k)$ telescopic and $c(G)=(c_2, \dots, c_k)$. For $H(S)$ the set of gaps of $S$, and for any function $f$ defined on $\mathbb{N}_0$,
\begin{equation}\label{eqn:explicit-free}
\sum\limits_{n\in H(S)}[f(n+g_1)-f(n)] = 
\sum\limits_{n_2=0}^{c_2-1}
\dots \sum\limits_{n_k=0}^{c_k-1} f\left(\sum\limits_{j=2}^k n_jg_j\right)  - \sum\limits_{n=0}^{g_1-1}f(n).
\end{equation}
\end{cor}
As before, with the function $f(n)=n$ in Equation \eqref{eqn:explicit-free} we recover the genus formula given in Equation \eqref{eqn:genus-free}.

\subsection{Prior results}
The motivating question for this paper is whether a free numerical semigroup has a minimal telescopic generating sequence. 
We can answer the question immediately for embedding dimension at most 3.

If $e(S)=1$, then $S=\mathbb{N}_0=\langle G\rangle$ for $G=(1)$, which is a minimal telescopic sequence.

If $e(S)=2$, then $S=\langle G\rangle$ for $G=(a,b)$ with $a,b>1$ and $\gcd(a,b)=1$. The sequence $G=(a,b)$ is telescopic, so $S$ is generated by a minimal telescopic sequence.

If $e(S)=3$, then we need a result of Herzog \cite{Herzog1970}. (We quote the result as written by Rosales and Garc\'ia-S\'anchez \cite{RosalesGarciaSanchez09}.)
\begin{thm}[{\cite[Cor.\ 9.5]{RosalesGarciaSanchez09}}]
For $S$ a numerical semigroup with $e(S)= 3$, $S$ is symmetric if and only if $S=\langle G\rangle$, for $G=(am_1, am_2, bm_1+cm_2)$ where $a, b, c, m_1, m_2\in\mathbb{N}$ with $\gcd(m_1, m_2)=1$, $a, m_1, m_2\ge2$, $b+c\ge2$, and $\gcd(a, bm_1+cm_2)=1$.
\end{thm}
For our purposes, if $S$ is a free numerical semigroup then $S$ is symmetric. If we also have $e(S)=3$, then by the above theorem we have $S=\langle G\rangle$ for $G=(am_1, am_2, bm_1+cm_2)$, which is necessarily telescopic (and of course minimal).

This approach will not work when $e(S)\ge4$, however, because there are symmetric numerical semigroups which are not free. For instance, the numerical semigroup $S=\langle e+1,\dots,e+e\rangle$ of embedding dimension $e$ is symmetric and not free for all $e\ge4$.

In what follows, we will take a different approach to show that every free numerical semigroup $S$ is generated by a minimal telescopic sequence.

\section{Explicit form of telescopic sequences}\label{sec:explicit-telescopic-construction}
In this section, our goal is to obtain an explicit form for the terms in a telescopic sequence, which we will need for our main result in Section \ref{sec:reduction-telescopic-minimal}. We will see a method to construct a telescopic sequence $G$ with $\gcd(G)=d$ given a desired sequence $c(G)$ and $d\in\mathbb{N}$. Our method uses the same ideas as in ``gluing'' of numerical semigroups, described by Watanabe \cite{Watanabe1973} and by Rosales and Garc\'ia-S\'anchez \cite[Chap.\ 8]{RosalesGarciaSanchez09}.

\subsection{Notation and preliminaries}
For notation, given $c_{m+1}, \dots, c_n\in\mathbb{N}_0$, let 
\begin{equation}
C_{m,n}=\prod_{j=m+1}^n c_j.
\end{equation}
If $n\le m$, then this product is empty, so $C_{m,n}=1$. In particular, we will make use of the fact that $C_{m,m}=1$. Additionally, as was mentioned in the introduction, for $i=1, \dots, k$, let $G_i=(g_1, \dots, g_i)$. Finally, recall that we will only consider sequences $G$ with $g_1+g_2>0$.

We begin with a few lemmas.

\begin{lemma}\label{lem:gcd-and-Cik}
Let $G\in\mathbb{N}_0^k$ be any sequence with $c(G)=(c_2, \dots, c_k)$. Then $\gcd(G_i)=C_{i,k}\gcd(G)$ for all $i=1, \dots, k$. 
In particular, if $\gcd(G)=1$, then $\gcd(G_i)=C_{i,k}$.
\end{lemma}
\begin{proof}
Observe that \[C_{i,k} = c_{i+1} c_{i+2} \cdots c_k = \frac{\gcd(G_i)}{\gcd(G_{i+1})}  \frac{\gcd(G_{i+1})}{\gcd(G_{i+2})}  \cdots  \frac{\gcd(G_{k-1})}{\gcd(G_k)}.\] 
With cancellation and the fact that $G_k=G$, we get $C_{i,k}=\gcd(G_i) / \gcd(G)$, as desired.
\end{proof}

For any sequence $G=(g_1, \dots, g_k)\in\mathbb{N}_0^k$ and any $m\in\mathbb{N}_0$, let $mG=(mg_1, \dots, mg_k)\in\mathbb{N}_0^k$. If $m\mid g_1, \dots, g_k$, let $G/m=(g_1/m, \dots, g_k/m)\in\mathbb{N}_0^k$. 

\begin{lemma}\label{lem:scaling-by-m}
For any $G\in\mathbb{N}_0^k$ and any $m\in\mathbb{N}_0$,
\begin{enumerate}
\item $(mG)_i = m(G_i)$; and
\item $\gcd((mG)_i)=m\gcd(G_i)$.
\end{enumerate}
\end{lemma}
\begin{proof}
Let $H=mG=(mg_1, \dots, mg_k)$. Then $H_i=(mg_1, \dots, mg_i)=m(g_1, \dots, g_i)=m(G_i)$, so $\gcd(H_i)=\gcd(m(G_i))=m\gcd(G_i)$.
\end{proof}
\begin{lemma}\label{lem:c-vals-and-scaling}
For any $G\in\mathbb{N}_0^k$ and any $m\in\mathbb{N}$,
\begin{enumerate}
\item $c(mG)=c(G)$; and
\item $G$ is telescopic if and only if $mG$ is telescopic.
\end{enumerate}
\end{lemma}
\begin{proof}
Let $c(G)=(c_2, \dots, c_k)$, $H=mG$, and $c(H)=(e_2, \dots, e_k)$. For $i=1, \dots, k$, we have $H_i=m(G_i)$. Thus, for $j=2, \dots, k$,
\begin{align*}
e_j&=\gcd(H_{j-1})/\gcd(H_j)\\
&=(m\gcd(G_{j-1}))/(m\gcd(G_j))\\
&=\gcd(G_{j-1})/\gcd(G_j)\\
&=c_j,
\end{align*} so $c(H)=c(G)$.

Finally, for any $j=2, \dots, k$, suppose $c_jg_j\in\langle G_{j-1}\rangle$. For $m\in\mathbb{N}$, this occurs exactly when $c_jmg_j\in\langle mG_{j-1}\rangle$. Since $e_j=c_j$ and $h_j=mg_j$, we conclude that $c_jg_j\in\langle G_{j-1}\rangle$ if and only if $e_jh_j\in\langle H_{j-1}\rangle$. This holds for all $j=2, \dots, k$, so $G$ is telescopic if and only if $H$ is telescopic.
\end{proof}

Let $S_k$ denote the symmetric group on $k$ letters. Elements of $S_k$ act on sequences $G\in\mathbb{N}_0^k$ in a natural way: for $\sigma\in S_k$ and $G=(g_1,\dots,g_k)$, let $\sigma(G)=(g_{\sigma(1)}, \dots, g_{\sigma(k)})\in\mathbb{N}_0^k$. Since $\sigma(G)$ is a permutation of $G$, we have $\langle \sigma(G)\rangle = \langle G\rangle$. In the following proposition, we show that the permutation $(1\;2)\in S_k$ takes telescopic sequences to telescopic sequences.

\begin{prop}\label{prop:swap-first-two}
For $k\ge2$, $G=(g_1, \dots, g_k)\in\mathbb{N}_0^k$, and $\sigma=(1\;2)\in S_k$, $G$ is telescopic if and only if $\sigma(G)$ is telescopic.
\end{prop}

\begin{proof}
Suppose $G$ is telescopic and let $H=\sigma(G)$. Then $H=(h_1, \dots, h_k)=(g_2, g_1, g_3, \dots, g_k)$. Let $c(H)=(e_2, \dots, e_k)$.

We have $e_2=\gcd(H_1)/\gcd(H_2)=g_2/\gcd(g_2, g_1)$, so \[e_2h_2=\frac{g_2}{\gcd(g_2, g_1)}g_1=\frac{g_1}{\gcd(g_1, g_2)}g_2.\] Since $g_1/\gcd(g_1, g_2)\in\mathbb{N}_0$, $e_2h_2\in\langle g_2\rangle=\langle H_1\rangle$.

For $i>1$, $H_i$ is a permutation of $G_i$, so $\langle H_i\rangle=\langle G_i\rangle$ and $\gcd(H_i)=\gcd(G_i)$.  For $j>2$, $e_j=\gcd(H_{j-1})/\gcd(H_j)=\gcd(G_{j-1})/\gcd(G_j)=c_j$, and therefore $e_jh_j=c_jg_j$. Since $G$ is telescopic, $c_jg_j\in\langle G_{j-1}\rangle$, so $e_jh_j\in\langle G_{j-1}\rangle=\langle H_{j-1}\rangle$. Thus $e_jh_j\in\langle H_{j-1}\rangle$ for $j=2, \dots, k$, so $H$ is a telescopic sequence. 

For the reverse implication, since $\sigma$ has order two, if $\sigma(G)$ is telescopic, then $\sigma(\sigma(G))=G$ is telescopic as well.
\end{proof}

\begin{rmk}\label{rmk:first-term-nonzero}
Since we already have the restriction that $g_1+g_2>0$, and since $G$ is telescopic exactly when $(1\;2)(G)$ is telescopic, for the rest of this paper we will assume $g_1>0$. Therefore $\gcd(g_1,\dots,g_i)>0$ for all $i\ge1$, and since $c_j=\gcd(g_1, \dots, g_{j-1})/\gcd(g_1, \dots, g_j)$, this will ensure that $c_j>0$ for all $j\ge2$ as well. We can therefore divide by $g_1$, $\gcd(G_i)$ for all $i$, and $c_j$ for all $j$ without worrying about possibly dividing by zero.
\end{rmk}

Finally, we show that the sequence consisting of the first $m$ terms of a telescopic sequence forms a telescopic sequence on its own.

\begin{lemma}\label{lem:shortened-telescopic}
For any $k\in\mathbb{N}$, let $G\in\mathbb{N}_0^k$ be a telescopic sequence with $g_1>0$ and $c(G)=(c_2, \dots, c_k)$. For any $m\in\{1, \dots, k\}$, the sequence $G_m=(g_1, \dots, g_m)$ is telescopic with $c(G_m)=(c_2, \dots, c_m)$. 
\end{lemma}
\begin{proof}
Let $m\in\{1, \dots, k\}$, let $H=G_m$, and let $c(H)=(e_2, \dots, e_m)$. Then $H=(h_1, \dots, h_m)$ with $h_i=g_i$ for $i=1, \dots, m$. Observe that $h_1>0$. We have $e_j=\gcd(H_{j-1})/\gcd(H_j)=\gcd(G_{j-1})/\gcd(G_j)=c_j$ for $2\le j\le m$. Since $c_jg_j\in\langle G_{j-1}\rangle$ for $j=2, \dots, k$, and since $m\le k$, it follows that $e_j h_j\in\langle H_{j-1}\rangle$ for $j=2, \dots, m$. Therefore $H$ is telescopic with $c(H)=(c_2, \dots, c_m)$.
\end{proof}

\subsection{Construction of telescopic sequences}
We can now give an explicit description for the elements of a telescopic sequence. We do so first in the case where $\gcd(G)=1$ and then for any value of $\gcd(G)$. 

\begin{prop}\label{prop:explicit-description-telescopic}
Suppose $G=(g_1, \dots, g_k)\in\mathbb{N}_0^k$ with $\gcd(G)=1$ and $c(G)=(c_2, \dots, c_k)$. Let $z_1=1$. Then $G$ is a telescopic sequence if and only if, for each $i=2, \dots, k$, there exists $z_i\in\mathbb{N}_0$ such that $g_i=z_i C_{i,k}$, $\gcd(z_i, c_i)=1$, and \[z_i\in\left\langle z_j C_{j,i-1} : 1\le j<i \right\rangle.\]
\end{prop}
\begin{proof}
$(\implies)$ Suppose $G$ is telescopic. For $i=2, \dots, k$, we will use strong induction.

We first show the statement is true for the base case $i=2$. Since $G$ is telescopic we have $c_2 g_2 \in\langle g_1 \rangle$. Let $z_2 = c_2 g_2 / g_1 \in\langle 1\rangle=\langle z_1 C_{1,1}\rangle$ since $z_1=C_{1,1}=1$. Then, since $g_1 = C_{1,k}$ (by Lemma \ref{lem:gcd-and-Cik}), we get $g_2 = z_2 g_1 / c_2 = z_2 C_{2,k}$. By definition $c_2 = g_1 / \gcd(g_1, g_2)$, so $C_{2,k}= g_1 / c_2 = \gcd(g_1, g_2) = \gcd(C_{1,k},  z_2 C_{2,k}) = \gcd(c_2 C_{2,k}, z_2 C_{2,k}) = C_{2,k} \gcd(c_2, z_2)$. Hence $\gcd(c_2, z_2)=1$. The conditions are therefore satisfied for $i=2$.

For strong induction, for $i=2, \dots, n$, with $n<k$, we assume $g_i = z_i C_{i,k}$ with $\gcd(z_i, c_i)=1$ and $z_i\in\langle z_j C_{j,i-1}: 1\le j<i\rangle$. By Lemma \ref{lem:gcd-and-Cik}, $g_1=C_{1,k}$ so $C_{n,k}\mid g_1$.  Then, since $C_{n,k} \mid C_{j,k}$ for $j=2, \dots, n$, by induction $C_{n,k}\mid g_2, \dots, g_n$. Since $G$ is telescopic, we have 
$c_{n+1}g_{n+1}\in\langle g_1, \dots, g_n\rangle$. We can divide through by $C_{n,k}$ (using the fact that $g_j/C_{n,k}=z_j C_{j,n}$) to get
\[\frac{c_{n+1}g_{n+1}} {C_{n,k}} = \frac{g_{n+1}} {C_{n+1,k}} \in \langle z_1 C_{1,n}, \dots, z_n C_{n,n} \rangle\subseteq\mathbb{N}_0.\] Let $z_{n+1} = g_{n+1} / C_{n+1,k}$. Then $g_{n+1} = z_{n+1} C_{n+1,k}$, and $z_{n+1}\in\langle z_j C_{j,n} : 1\le j< n+1 \rangle$.

Next, we verify the gcd condition. By Lemma \ref{lem:gcd-and-Cik}, $\gcd(G_n)=C_{n,k}$, so
\[c_{n+1} = \frac{\gcd(G_n)} {\gcd(G_{n+1})} = \frac{C_{n,k}} {\gcd(C_{n,k},  g_{n+1})}.\]
Therefore
\begin{align*}
C_{n+1,k}&=C_{n,k}/c_{n+1}\\
&=\gcd(C_{n,k}, g_{n+1})\\
&=\gcd(C_{n,k}, z_{n+1}C_{n+1,k}) \\
&= C_{n+1,k}\gcd(c_{n+1}, z_{n+1}),
\end{align*}
so $\gcd(c_{n+1}, z_{n+1})=1$. 

Since we have verified all of the conditions for $i=n+1$, by induction the statement is true for all $i=2, \dots, k$.

$(\impliedby)$
Suppose $G=(g_1, g_2, \dots, g_k)=(C_{1,k}, z_2 C_{2,k}, \dots, z_k C_{k,k})$ with $\gcd(z_i, c_i)=1$ and $z_i \in \langle z_j C_{j, i-1} : 1\le j<i \rangle$ for $i=2, \dots, k$. 
For $G$ to be telescopic, we need $c_i g_i \in \langle G_{i-1} \rangle$ for $i=2, \dots, k$. Since $c_i g_i=c_i z_i C_{i,k} = z_i C_{i-1,k}$ and $z_i\in\langle z_jC_{j,i-1} : 1\le j<i\rangle$, we have 
\begin{align*}
c_i g_i = z_i C_{i-1, k} & \in\langle z_j C_{j, i-1} C_{i-1, k} : 1\le j<i \rangle \\
& =\langle z_j C_{j,k} : 1\le j<i \rangle\\
& =\langle g_1, \dots, g_{i-1} \rangle\\
& = \langle G_{i-1}\rangle
\end{align*}
for $i=2, \dots, k$. Thus, $G$ is telescopic.
\end{proof}

We now give the result for any $\gcd(G)$.

\begin{cor}\label{cor:explicit-description-telescopic}
Suppose $G=(g_1, \dots, g_k)\in\mathbb{N}_0^k$ and $c(G)=(c_2, \dots, c_k)$. Let $z_1=\gcd(G)$. Then $G$ is a telescopic sequence if and only if, for each $i=2, \dots, k$, there exists $z_i\in\langle z_1\rangle$ such that $g_i=z_iC_{i,k}$, $\gcd(z_i/z_1, c_i)=1$, and \[z_i\in\langle z_jC_{j,i-1}:1\le j<i\rangle.\]
\end{cor}
\begin{proof}
Let $d=\gcd(G)$ and $H=G/d$, so $\gcd(H)=1$. By Lemma \ref{lem:c-vals-and-scaling}, $c(H)=c(G)=(c_2, \dots, c_k)$. Let $y_1=1$. For $H=(h_1, \dots, h_k)$, by Proposition \ref{prop:explicit-description-telescopic}, $H$ is telescopic if and only if, for each $i=2, \dots, k$, there exists $y_i\in\mathbb{N}_0$ such that $h_i=y_iC_{i,k}$, $\gcd(y_i, c_i)=1$, and $y_i\in\langle y_jC_{j,i-1}:1\le j<i\rangle$. Since $G=dH$, $g_i=dh_i$. By Lemma \ref{lem:c-vals-and-scaling}, $G$ is telescopic if and only if $H$ is telescopic. Finally, let $z_i=dy_i$.

We rewrite the quoted result of Proposition \ref{prop:explicit-description-telescopic} to now say that with $z_1=dy_1=d$, $G$ is telescopic if and only if, for each $i=2, \dots, k$, there exists $z_i\in  d\mathbb{N}_0$ such that $g_i=z_iC_{i,k}$, $\gcd(z_i/d, c_i)=1$, and $z_i\in\langle z_jC_{j,i-1}:1\le j<i\rangle$. Since $z_1=d$, the corollary statement follows.
\end{proof}

\begin{rmk}\label{rmk:how-to-construct-telescopic}
For $k\in\mathbb{N}$, given any sequence $(c_2,\dots,c_k)\in\mathbb{N}^{k-1}$ and any $d\in\mathbb{N}$, we have a process to construct any telescopic sequence $G\in\mathbb{N}_0^k$ for which $\gcd(G)=d$ and $c(G)=(c_2,\dots,c_k)$. We just need non-negative integers $z_1, \dots, z_k$ where $z_1=d$ and, for $i=2,\dots,k$, $\gcd(z_i, d c_i)=d$ and $z_i \in \langle z_j C_{j,i-1} : 1\le j<i\rangle$. Then, for $i=1,\dots,k$, we let $g_i=z_iC_{i,k}$ to produce $G=(g_1,\dots,g_k)$, a sequence for which $\gcd(G)=d$ and $c(G)=(c_2,\dots,c_k)$.
\end{rmk}

\begin{ex}\label{ex:telescopic-construction}
Suppose we want a telescopic sequence $G\in\mathbb{N}_0^5$ with $\gcd(G)=4$ and $c(G)=(c_2,c_3,c_4,c_5)=(3, 2, 5, 3)$.

To start, we have $z_1=\gcd(G)=4$.

For $z_2,z_3,z_4,z_5$, we have some choices to make. For $z_2$, we need $\gcd(z_2,4\cdot 3)=4$ and $z_2\in\langle z_1\rangle$. For $z_3$, we need $\gcd(z_3,4\cdot 2)=4$ and $z_3\in\langle 3z_1,z_2 \rangle$. For $z_4$, we need $\gcd(z_4,4\cdot 5)=4$ and $z_4\in\langle 6z_1, 2z_2, z_3\rangle$. For $z_5$, we need $\gcd(z_5,4\cdot 3)=4$ and $z_5\in\langle 30z_1,10z_2,5z_3,z_4\rangle$.

One option is to take $z_2=z_3=z_4=z_5=4$. With $g_i=z_iC_{i,k}$, we get $G=(g_1,g_2,g_3,g_4,g_5)=(360,120,60,12,4)$.

Another option is to take $z_2=8$, $z_3=20$, $z_4=28$, $z_5=44$. With $g_i=z_iC_{i,k}$, we get $G=(360,240,300,84,44)$.

Both sequences are telescopic with $\gcd(G)=4$ and $c(G)=(3,2,5,3)$.
\end{ex}

Observe that the first sequence in Example \ref{ex:telescopic-construction} is not minimal, whereas the second sequence is. In Section \ref{sec:explicit-minimal-telescopic-construction}, we present Corollary \ref{cor:explicit-min}, a refinement of Remark \ref{rmk:how-to-construct-telescopic}, which allows us to construct any minimal telescopic sequence. We postpone its appearance because it relies on a result (Proposition \ref{prop:two-cases}) from Section \ref{sec:reduction-telescopic-minimal}.

\subsection{Specific values for certain telescopic sequences}\label{subsec:specific-values}
We describe the sequences $(c_2, \dots, c_k)$ and $(z_1, \dots, z_k)$ that occur for some families of telescopic sequences: geometric, supersymmetric, and compound sequences. Gassert and Shor \cite{GassertShor17} detail some applications of Corollary \ref{cor:tuenter-telescopic-id} to these sequences, as well as connections to certain algebraic curves.

A (finite) geometric sequence $G$ of length $k$ with $\gcd(G)=1$ is a sequence of the form $(g_1, \dots, g_k)$ where $g_i=a^{k-i}b^{i-1}$ for $a, b \in \mathbb{N}$ with $\gcd(a, b)=1$. In this case, $\gcd(g_1, \dots, g_i)=a^{k-i}$, so $c_i=a$ for $i=2, \dots, k$. In the notation of Proposition \ref{prop:explicit-description-telescopic}, we have $z_i=b^{i-1}$ for $i=1, \dots, k$.

As named and studied by Fr\"oberg, Gottlieb, and H\"aggkvist \cite{FrobergGottliebHaggkvist87}, a \emph{supersymmetric sequence} $G$ of length $k$ with $\gcd(G)=1$ is a sequence of the form $(g_1, \dots, g_k)$ where $g_i =A / a_i$ for pairwise coprime natural numbers $a_1, \dots, a_k$ and $A=a_1 \cdots a_k$. In this case, $\gcd(g_1, \dots, g_i) = A/(a_1 \cdots a_i)$, so $c_i = a_i$. In the notation of Proposition \ref{prop:explicit-description-telescopic}, we have $z_1=1$ and $z_i = a_1 \cdots a_{i-1}$ for $i = 2, \dots, k$.

Geometric and supersymmetric sequences are special cases of compound sequences, which were named and studied by Kiers, O'Neill, and Ponomarenko \cite{KiersONeillPonomarenko16}. A \emph{compound sequence} $G$ of length $k$ with $\gcd(G)=1$ is a sequence of the form $(g_1, \dots, g_k)$ where $g_i = b_1 \cdots b_{i-1} a_i \cdots a_{k-1}$ for $a_i, b_j \in \mathbb{N}$ with $\gcd(a_i, b_j)=1$ for all $i\ge j$. In this case, $\gcd(g_1, \dots, g_i) = a_{i}\cdots a_{k-1}$, so $c_i = a_{i-1}$. In the notation of Proposition \ref{prop:explicit-description-telescopic}, we have $z_1=1$ and $z_i = b_1 \cdots b_{i-1}$ for $i=2, \dots, k$.

\section{Operations on telescopic sequences}\label{sec:sequence-operations}
In this section, we will introduce two operations: $\rho_n$, which maps a sequence of length $k$ to a sequence of length $k-1$; and $\tau_{g,m}$, a gluing map as described by Rosales and Garc\'ia-S\'anchez \cite[Chap.\ 8]{RosalesGarciaSanchez09}, which maps a sequence of length $k$ to a sequence of length $k+1$. With appropriate parameters, these operations ``preserve telescopicness,'' which is to say they map telescopic sequences to telescopic sequences.  We will also show that any function between two telescopic sequences with the same greatest common divisor can be written as a composition of these functions. In particular, $\rho_n$ will be useful in Section \ref{sec:reduction-telescopic-minimal} where we will use it to eliminate redundant terms from a telescopic sequence which is not minimal.

As a reminder, in Remark \ref{rmk:first-term-nonzero} we saw that we can assume $g_1>0$ for any sequence $G=(g_1,\dots,g_k)\in\mathbb{N}_0^k$. We will continue to make this assumption, from which it follows that $\gcd(G_i)>0$ for all $i\ge1$ and $c_j>0$ for all $j\ge2$.

\subsection{Two useful sequence operations}
For $k,l\in\mathbb{N}$, let $G=(g_1, \dots, g_k)\in\mathbb{N}_0^k$ and $H=(h_1, \dots, h_l)\in\mathbb{N}_0^l$. 
For integers $i$ and $j$ with $0\le i < j\le k$, let $G_{i,j}=(g_{i+1}, g_{i+2}, \dots, g_j)\in\mathbb{N}_0^{j-i}$. 
Let $G\times H=(g_1, \dots, g_k, h_1, \dots, h_l)\in\mathbb{N}_0^{k+l}$.

Using this notation, we have the following results.
\begin{prop}\label{prop:basic-props}
For any $G\in\mathbb{N}_0^k$,
\begin{enumerate}
\item $\gcd(G\times H)=\gcd\left(\gcd(G), \gcd(H)\right)$ for any $H\in\mathbb{N}_0^l$; and
\item $G_{i,j}\times G_{j,l}=G_{i,l}$ for any integers $i$, $j$, $l$ with $0\le i<j<l\le k$.
\end{enumerate}
\end{prop}

\begin{proof}
The first result is a consequence of the fact that for $a, b, c\in\mathbb{Z}$, $\gcd(a, b, c)=\gcd(\gcd(a, b), c)$. 

The second result is equivalent to saying $(g_{i+1}, \dots, g_j)\times (g_{j+1}, \dots, g_l)=(g_{i+1}, \dots, g_l)$, which is true when $i<j<l$.
\end{proof}

We now introduce two operations on sequences which preserve telescopicness. The first, $\rho_n$, produces a sequence with one fewer entry. The second, $\tau_{g,m}$, produces a sequence that has one more entry. As we will see, given any two telescopic sequences $G$ and $H$ with $\gcd(G)=\gcd(H)$, one can compose finitely many of these functions together to transform $G$ into $H$.

For the definition of $\rho_n$ below, note that for $G\in\mathbb{N}_0^k$, $c(G)=(c_2, \dots, c_k)$, and $n=2, \dots, k$, we have $c_n\mid C_{n-1,k}$. Then, by Lemma \ref{lem:gcd-and-Cik}, $c_n\mid \gcd(G_{n-1})$, so $G_{n-1}/c_n\in\mathbb{N}_0^{n-1}$.
\begin{defn}
For any $k\ge2$, $G\in\mathbb{N}_0^k$, $c(G)=(c_2, \dots, c_k)$, and $n=2, \dots, k$, let 
\begin{equation}
\rho_n(G)=\left(G_{n-1}/c_n\right)\times G_{n,k}\in\mathbb{N}_0^{n-1}\times\mathbb{N}_0^{k-n}=\mathbb{N}_0^{k-1}.
\end{equation}
In other words, $\rho_n(g_1,\dots,g_k)=(g_1/c_n,\dots,g_{n-1}/c_n,g_{n+1},\dots,g_k)$.
\end{defn}
\begin{rmk}
We specify $k\ge2$ in the definition because there are no corresponding $n$ for which to define $\rho_n$ when $k=1$.
\end{rmk}

\begin{rmk}\label{rmk:removing-first-entry}
The definition of $\rho_n$ allows us to remove the $n$th entry of $G$ for any $n\in\{2, \dots, k\}$. If we wish to remove the first entry, we can apply the permutation $\sigma=(1\;2)$ first and then apply $\rho_2$.
\[\rho_2(\sigma(G))=\rho_2(g_2, g_1, g_3, \dots, g_k)=(\gcd(g_2, g_1), g_3, \dots, g_k).\] Since $\gcd(g_2,g_1)=\gcd(g_1,g_2)$, it follows that $\rho_2(\sigma(G))=\rho_2(G)$. Thus, removal of the first entry is equivalent to removal of the second entry.
\end{rmk}

\begin{defn}
For any $k\ge1$, $G\in\mathbb{N}_0^k$, $g\in \langle G\rangle$, and $m\in\mathbb{N}$ with $\gcd(m, g)=1$, let 
\begin{equation}
\tau_{g,m}(G)=(mG)\times(g)\in\mathbb{N}_0^{k+1}.
\end{equation}
In other words, $\tau_{g,m}(g_1,\dots,g_k)=(mg_1,\dots,mg_k,g)$.
\end{defn}

\begin{ex}
As in Example \ref{ex:question-example}, let $G=(660,550,352,50,201)\in\mathbb{N}_0^5$. We have $c(G)=(6,5,11,2)$.

We can apply $\rho_n$ for any $n\in\{2,3,4,5\}$. Applying $\rho_2$ (and thereby removing $g_2$), we get $\rho_2(G)=(G_1/c_2)\times G_{2,5}=(110,352,50,201)\in\mathbb{N}_0^4$. If we wish to remove $g_1$, we apply the permutation $(1\;2)$ first. Let $H=(1\;2)(G)=(550,660,352,50,201)$. Then $c(H)=(e_2,e_3,e_4,e_5)=(5,6,11,2)$ and 
\[ 
\rho_2((1\;2) (G))=\rho_2(H)=(H_1/e_2)\times H_{2,5}=(110,352,50,201)\in\mathbb{N}_0^4,
\]
which, as expected, is equal to $\rho_2(G)$.

We can apply $\tau_{g,m}$ for any $g\in\langle G\rangle$ and any $m$ with $\gcd(m,g)=1$. Applying $\tau_{251,3}$, we have $\tau_{251,3}(G)=(3G)\times(251)=(1980,1650,1056,150,603,251)\in\mathbb{N}_0^6$.
\end{ex}

Our goal with $\rho_n$ and $\tau_{g,m}$ is to determine the conditions for which that they send telescopic sequences to telescopic sequences. We begin by investigating their basic properties. In Lemma \ref{lem:rho-tau-mults}, we see how $\rho_n$ and $\tau_{g,m}$ act on multiples of sequences. In Lemma \ref{lem:rho-c-vals}, we compute $\gcd(\rho_n(G))$ and $c(\rho_n(G))$. In Lemma \ref{lem:tau-c-vals}, we compute $\gcd(\tau_{g,m}(G))$ and $c(\tau_{g,m}(G))$.

\begin{lemma}\label{lem:rho-tau-mults}
Suppose $G\in\mathbb{N}_0^k$ and $d\in\mathbb{N}$. 
For $n=2, \dots, k$, $\rho_n(dG)=d\rho_n(G)$. 
For any $g\in\langle G\rangle$ and any $m\in\mathbb{N}$ with $\gcd(g, m)=1$, $\tau_{dg,m}(dG)=d\tau_{g,m}(G)$.
\end{lemma}

\begin{proof}
Let $c(G)=(c_2, \dots, c_k)$. By Lemma \ref{lem:c-vals-and-scaling}, $c(dG)=c(G)$, so $c(dG)=(c_2, \dots, c_k)$. Then \begin{align*}
\rho_n(dG)
&= \left((dG)_{n-1}/c_n\right)\times \left((dG)_{n,k} \right)\\
&= \left(d(G_{n-1}/c_n)\right)\times \left(d(G_{n,k})\right)\\
&= d\left((G_{n-1}/c_n) \times G_{n,k}\right)\\
&= d\rho_n(G).
\end{align*}

Next,
$\tau_{dg,m}(dG)=(mdG)\times(dg)=d\left((mG)\times(g)\right)=d\tau_{g,m}(G)$.
\end{proof}

\begin{lemma}\label{lem:rho-c-vals}
Let $G\in\mathbb{N}_0^k$ with $c(G)=(c_2, \dots, c_k)$. For any $n=2, \dots, k$, we have $\gcd(\rho_n(G))=\gcd(G)$ and $c(\rho_n(G))=(c_2, \dots, c_{n-1}, c_{n+1}, \dots, c_k)$. 
\end{lemma}

\begin{proof}
We begin with the case where $\gcd(G)=1$.

Let $H=\rho_n(G)$, so $H=(h_1, \dots, h_{k-1})$ where, for $i=1, \dots, n-1$, we have $h_i=g_i/c_n$, and for $i=n, \dots, k-1$, we have $h_i=g_{i+1}$. Let $c(H)=(e_2, \dots, e_{k-1})$.

We begin by computing $\gcd(H_i)$. For $i<n$, $H_i=G_i/c_n$. Thus $\gcd(H_i)=\gcd\left(G_i/c_n\right)$. By Lemma \ref{lem:gcd-and-Cik}, $\gcd(G_i)=C_{i,k}=c_{i+1}\cdots c_k$. Since $i<n$, $\gcd(G_i)$ contains $c_n$ as a factor. Therefore $\gcd\left(G_i/c_n\right)=C_{i,k}/c_n$.

For $i\ge n$, we wish to show $\gcd(H_i)=C_{i+1,k}$. We proceed with induction on $i$. If $i=n$, $H_n=(G_{n-1}/c_n)\times(g_{n+1})$. By the previous paragraph, $\gcd(G_{n-1}/c_n) = C_{n,k}$. By Proposition \ref{prop:explicit-description-telescopic}, $g_{n+1}=z_{n+1} C_{n+1,k}$ with $\gcd(c_{n+1}, z_{n+1})=1$. Then 
\begin{align*}
\gcd(H_n)
&=\gcd\left((G_{n-1}/c_n)\times (g_{n+1})\right)\\
&=\gcd\left(\gcd(G_{n-1}/c_n), g_{n+1}\right)\\
&=\gcd(C_{n,k}, g_{n+1})\\
&=\gcd(c_{n+1} C_{n+1,k}, z_{n+1} C_{n+1,k})\\
&=C_{n+1,k} \gcd(c_{n+1}, z_{n+1})\\
&=C_{n+1,k}.
\end{align*}

For $i>n$, by Proposition \ref{prop:explicit-description-telescopic} we have $g_{i+2}=z_{i+2}C_{i+2,k}$ with $\gcd(c_{i+2}, z_{i+2})=1$. By induction we assume $\gcd(H_i)=C_{i+1,k}$, so
\begin{align*}
\gcd(H_{i+1})&=\gcd(h_1, \dots, h_{i+1})\\
&=\gcd(\gcd(H_i), g_{i+2})\\
&=\gcd(C_{i+1,k}, z_{i+2}C_{i+2,k})\\
&=\gcd(c_{i+2}C_{i+2,k}, z_{i+2}C_{i+2,k})\\
&=C_{i+2,k}.
\end{align*}
Therefore $\gcd(H_i)=C_{i+1,k}$ for $i>n$. In particular, $\gcd(H)=\gcd(H_{k-1})=C_{k,k}=1$.

Now we compute $e_j$. For $j<n$, $e_j=\gcd(H_{j-1})/\gcd(H_j)=(C_{j-1,k}/c_n)/(C_{j,k}/c_n)=c_j$. For $j=n$, $e_j=\gcd(H_{n-1})/\gcd(H_n)=(C_{n-1,k}/c_n)/C_{n+1,k}=c_{n+1}=c_{j+1}$. For $j>n$, $e_j=\gcd(H_{j-1})/\gcd(H_j)=C_{j,k}/C_{j+1,k}=c_{j+1}$. Therefore $c(H)=(c_2, \dots, c_{n-1}, c_{n+1}, \dots, c_k)$, as desired.

Now, suppose $\gcd(G)=d\ge1$, let $G'=G/d$. Since $c(G)=c(G')$, we have $c(G')=(c_2, \dots, c_k)$. Since $\gcd(G')=1$, we apply the above result to obtain $\gcd(\rho_n(G'))=1$ and $c(\rho_n(G'))=(c_2, \dots, c_{n-1},  c_{n+1}, \dots, c_k)$. By Lemma \ref{lem:rho-tau-mults}, $\rho_n(G)=\rho_n(dG')=d\rho_n(G')$, so $\gcd(\rho_n(G))=\gcd(d\rho_n(G'))=d\gcd(\rho_n(G'))=d$ and 
$c(\rho_n(G))=c(d\rho_n(G'))=c(\rho_n(G'))=(c_2, \dots, c_{n-1}, c_{n+1}, \dots, c_k)$.
\end{proof}

\begin{lemma}\label{lem:tau-c-vals}
Let $G\in\mathbb{N}_0^k$ with $c(G)=(c_2, \dots, c_k)$. For any $g\in\langle G\rangle$ and any $m\in\mathbb{N}$ with $\gcd(g, m)=1$, we have $\gcd(\tau_{g,m}(G))=\gcd(G)$ and $c(\tau_{g,m}(G))=c(G)\times(m)$.
\end{lemma}

\begin{proof}
Let $H=\tau_{g,m}(G)$, and let $c(H)=(e_2, \dots, e_{k+1})$. We'll compute $e_j$ for $j=2, \dots, k$, and for $j=k+1$.

We have $H_k=mG$. By Lemma \ref{lem:c-vals-and-scaling}, $c(mG)=c(G)$. By Lemma \ref{lem:shortened-telescopic}, $c(H_k)=(e_2,\dots,e_k)$. Thus $e_j=c_j$ for $j=2,\dots,k$.

For $j=k+1$, first note that
\begin{align*}
\gcd(H)&=\gcd((mG)\times(g))\\
&=\gcd(\gcd(mG), g)\\
&=\gcd(m\gcd(G), g)\\
&=\gcd(\gcd(G), g)\\
&=\gcd(G),
\end{align*}
where the fourth equality follows from $\gcd(g,m)=1$ and the final equality follows from the fact that $g\in\langle G\rangle$, which implies that $\gcd(G)\mid g$. 
Then 
\begin{align*}
e_{k+1}&=\gcd(H_k)/\gcd(H_{k+1})\\
&=\gcd(mG_k)/\gcd(H)\\
&=m\gcd(G)/\gcd(G)\\
&=m.\end{align*}

Combining the results for $j=2,\dots,k$ and $j=k+1$, we have $c(H)=(c_2, \dots, c_k, m)=c(G)\times(m)$, as desired.
\end{proof}

With appropriate parameters, $\rho_n$ and $\tau_{g,m}$ are inverses of each other as the following lemma shows.

\begin{lemma}\label{lem:telescopic-operations-composition}
Let $G=(g_1, \dots, g_k)\in\mathbb{N}_0^k$ with $c(G)=(c_2, \dots, c_k)$, $g\in\langle G\rangle$, and $m\in\mathbb{N}$ with $\gcd(g, m)=1$. Then $\tau_{g_k,c_k}(\rho_k(G))=G$ and $\rho_{k+1}(\tau_{g,m}(G))=G$.
\end{lemma}

\begin{proof}
Since $\rho_k(G)=G_{k-1}/c_k$,
\begin{align*}
\tau_{g_k,c_k}(\rho_k(G)) 
&= \tau_{g_k,c_k}\left(G_{k-1}/c_k\right)\\
& = \left(c_k G_{k-1}/c_k\right)\times(g_k)\\
&= G_{k-1}\times(g_k)\\
&= G.
\end{align*}

By Lemma \ref{lem:tau-c-vals}, $c(\tau_{g,m}(G))=(c_2, \dots, c_k, m)$, so 
\begin{align*}
\rho_{k+1}(\tau_{g,m}(G))
& =\rho_{k+1}\left((mG)\times(g)\right)\\
& =\rho_{k+1}\left((mg_1, \dots, mg_k, g)\right) \\
& =(mg_1, \dots, mg_k)/m\\
& =(g_1, \dots, g_k)\\
&=G.
\end{align*}
\end{proof}

\subsection{Telescopicness-preserving operations}
Now that we have developed some of the properties of $\rho_n$ and $\tau_{g,m}$, we will see how they map telescopic sequences to telescopic sequences. The goal is to then construct a map between any two telescopic sequences with the same greatest common divisor using some sequence of these maps composed together.

We first show that $\rho_n$ preserves telescopicness.

\begin{prop}\label{prop:telescopic-operations-rho}
For $k\ge2$, let $G\in\mathbb{N}_0^k$. For $n=2, \dots, k$, if $G$ is telescopic, then $\rho_n(G)$ is telescopic.
\end{prop}

\begin{proof}
Let $H=\rho_n(G)$. Then $H=(h_1, \dots h_{k-1})$ where $h_i=g_i/c_n$ for $i<n$, and $h_i=g_{i+1}$ for $i\ge n$. Let $c(H)=(e_2, \dots, e_{k-1})$. By Lemma \ref{lem:rho-c-vals}, $c(H)=(c_2, \dots, c_{n-1}, c_{n+1}, \dots, c_k)$.
Since $G$ is telescopic, $c_jg_j\in\langle G_{j-1}\rangle$ for $j=2, \dots, k$. We need to show $e_jh_j\in\langle H_{j-1}\rangle$ for $j=2, \dots, k-1$.

For $i=1, \dots, n-1$, $H_i=G_i/c_n$, so for $j=2, \dots, n-1$, $e_j=c_j$. Since $G$ is telescopic, $c_jg_j\in\langle G_{j-1}\rangle$. Thus $e_jh_j=c_jg_j/c_n\in\langle G_{j-1}/c_n\rangle=\langle H_{j-1}\rangle$.

Next, we have $c_ng_n\in\langle G_{n-1}\rangle$. Since $c_n\mid\gcd(G_{n-1})$, $g_n\in\langle G_{n-1}/c_n\rangle$, so $\langle G_{n-1}/c_n\rangle=\langle (G_{n-1}/c_n)\times(g_n)\rangle$. For $i=n, \dots, k-1$, $\langle G_i\rangle\subset\langle (G_{n-1}/c_n) \times G_{n-1,i}\rangle= \langle (G_{n-1}/c_n) \times G_{n,i}\rangle=\langle H_{i-1}\rangle$. For $j=n, \dots, k-1$,  we have $e_jh_j=c_{j+1}g_{j+1}\in\langle G_j\rangle\subset\langle H_{j-1}\rangle$.

Since $e_jh_j\in\langle H_{j-1}\rangle$ for $j=2, \dots, k-1$, $H$ is telescopic.
\end{proof}

\begin{rmk}
Unfortunately, Proposition \ref{prop:telescopic-operations-rho} is not an ``if and only if'' statement. Since all sequences of length 2 are telescopic, we can illustrate this with a sequence of length 3 that is not telescopic. For example, consider the sequence $G=(3,4,5)$. We have that $c(G)=(c_2,c_3)=(3,1)$. Since $1\cdot 5\not\in\langle 3,4\rangle$, $G$ is not telescopic. However, $\rho_2(G)=(1,5)$ and $\rho_3(G)=(3,4)$, and both are telescopic sequences.
\end{rmk}

Next, we show that, with proper parameters, $\tau_{g,m}$ preserves telescopicness.

\begin{prop}\label{prop:telescopic-operations-tau}
For $k\in\mathbb{N}$, let $G\in\mathbb{N}_0^k$, $g\in\langle G\rangle$, and $m\in\mathbb{N}$ with $\gcd(g, m)=1$. 
Then $G$ is telescopic if and only if $\tau_{g,m}(G)$ is telescopic.
\end{prop}

\begin{proof}
($\implies$) 
Let $H=\tau_{g,m}(G)$, so $H=(h_1, \dots, h_{k+1})$ with $h_i=mg_i$ for $i=1, \dots, k$, and $h_{k+1}=g$. Let $c(H)=(e_2, \dots, e_{k+1})$. By Lemma \ref{lem:tau-c-vals}, $c(H)=(c_2, \dots, c_k, m)$. Since $G$ is telescopic, $c_jg_j\in\langle G_{j-1}\rangle$ for $j=2, \dots, k$. We need to show $e_jh_j\in\langle H_{j-1}\rangle$ for $j=2, \dots, k+1$.

For $j=2, \dots, k$, $e_jh_j=c_jg_j\in\langle G_{j-1}\rangle=\langle H_{j-1}\rangle$. For $j=k+1$, $e_jh_j=mg$. Since $g\in\langle G\rangle$, $mg\in\langle mG\rangle=\langle H_k\rangle$. Thus $e_{k+1}h_{k+1}\in\langle H_k\rangle$. Therefore $e_jh_j\in\langle H_{j-1}\rangle$ for $j=2, \dots, k+1$, so $H$ is telescopic.

($\impliedby$) 
For $G=(g_1,\dots,g_k)$, if $\tau_{g,m}(G)\in\mathbb{N}_0^{k+1}$ is telescopic, then $\rho_{k+1}(\tau_{g,m}(G))$ is also telescopic. Since $\rho_{k+1}(\tau_{g,m})(G)=\rho_{k+1}((mG)\times(g))=G$, and $\rho_{n}$ maps telescopic sequences to telescopic sequences, we conclude that $G$ is telescopic.
\end{proof}

Now that we know $\rho_n$ and $\tau_{g,m}$ send telescopic sequences to telescopic sequences (using appropriate parameters with $\tau_{g,m}$), in Lemma \ref{lem:rho-down} we will construct a map from $\mathbb{N}_0^k$ to $\mathbb{N}_0$ using a sequence of $\rho_n$ functions which sends $G$ to $\gcd(G)$. Following this, in Lemma \ref{lem:tau-up} we will construct a map from $\mathbb{N}_0$ to $\mathbb{N}_0^k$ using a sequence of $\tau_{g,m}$ functions which sends $\gcd(G)$ to $G$.

\begin{lemma}\label{lem:rho-down}
For any $G\in\mathbb{N}_0^k$ (telescopic or not), let 
\begin{equation}
R_G=\rho_2\circ\rho_3\circ\dots\circ\rho_k:\mathbb{N}_0^k\to\mathbb{N}_0.
\end{equation}
Then
$R_G(G)=\gcd(G)$.
\end{lemma}

\begin{proof}
If $k=1$, then $R_G$ is the identity map. Since $\gcd(g_1)=g_1$, $R_G(G)=\gcd(G)$.

Now, suppose $k>1$. Let $c(G)=(c_2, \dots, c_k)$. Then $\rho_k(G)=G_{k-1}/c_k$ and $c(\rho_k(G))=(c_2, \dots, c_{k-1})$. Then $\rho_{k-1}(\rho_k(G))=(G_{k-2}/c_k)/c_{k-1}$ and $c(\rho_{k-1}(\rho_k(G)))=(c_2, \dots, c_{k-2})$. Continuing in this way we find \[(\rho_2\circ\rho_3\circ\dots\circ\rho_k)(G)=
G_1/(c_k c_{k-1} \cdots c_2)=g_1/C_{1,k}=\gcd(G),\]
by Lemma \ref{lem:gcd-and-Cik}.
\end{proof}

\begin{lemma}\label{lem:tau-up}
For $G\in\mathbb{N}_0^k$ a telescopic sequence with $c(G)=(c_2, \dots, c_k)$ and $z_i=g_i/C_{i,k}$ for $i=1, \dots, k$, let 
\begin{equation}
T_G=\tau_{z_k,c_k}\circ \tau_{z_{k-1},c_{k-1}}\circ\dots\circ\tau_{z_2,c_2}:\mathbb{N}_0\to\mathbb{N}_0^k.
\end{equation}
Then $T_G(\gcd(G))=G$.
\end{lemma}

\begin{proof}
If $k=1$, then $T_G$ is the identity map. Since $\gcd(g_1)=g_1$, $T_G(\gcd(G))=G$.

Now, suppose $k>1$. We first consider the case where $\gcd(G)=1$. For $j=2, \dots, k$, we wish to show $\tau_{z_j,c_j}\circ\cdots\circ\tau_{z_2,c_2}(1)$ is defined and equal to $(z_1C_{1,j}, \dots, z_jC_{j,j})$.

For $j=2$, since $z_2\in\langle 1\rangle$, $\tau_{z_2,c_2}(1)=(c_2, z_2)=(z_1C_{1,2}, z_2C_{2,2})$.

Now, for induction assume $\tau_{z_j,c_j}\circ\cdots\circ\tau_{z_2,c_2}(1)$ is defined and equal to $(z_1C_{1,j}, \dots, z_jC_{j,j})$. Since $G$ is telescopic, $c_{j+1}g_{j+1}\in\langle G_j\rangle$, so $z_{j+1}C_{j,k}\in\langle z_1C_{1,k}, \dots, z_jC_{j,k}\rangle$. It follows that $z_{j+1}\in\langle z_1C_{1,j}, \dots, z_jC_{j,j}\rangle$, and therefore $z_{j+1}\in\langle\tau_{z_j,c_j}\circ\cdots\circ\tau_{z_2,c_2}(1)\rangle$. Then 
\begin{align*}
(\tau_{z_{j+1},c_{j+1}}\circ\cdots\circ\tau_{z_2,c_2})(1)
&= \tau_{z_{j+1}, c_{j+1}}(z_1C_{1,j}, \dots, z_jC_{j,j})\\
&= (z_1C_{1,j}c_{j+1}, \dots, z_jC_{j,j}c_{j+1}, z_{j+1})\\
&= (z_1C_{1,j+1}, \dots, z_jC_{j,j+1}, z_{j+1}C_{j+1,j+1}).
\end{align*}
In particular, for $j=k$, we have 
\[(\tau_{z_k,c_k}\circ \tau_{z_{k-1},c_{k-1}}\circ\dots\circ\tau_{z_2,c_2})(1)=(z_1 C_{1,k}, \dots, z_k C_{k,k})=G.\]

Now, suppose $\gcd(G)=d\ge 1$ and let $H=G/d$, so $H=(h_1, \dots, h_k)$ with $h_i=g_i/d$ and $c(H)=c(G)$. Let $y_i=z_i/d$. By Proposition \ref{prop:explicit-description-telescopic}, since $g_i=z_iC_{i,k}$ with $z_i\in d\mathbb{N}_0$ for $i=1, \dots, k$, we have $h_i=y_iC_{i,k}$ with $y_i\in\mathbb{N}_0$. By above, 
\[(\tau_{y_k,c_k}\circ \tau_{y_{k-1},c_{k-1}}\circ\dots\circ\tau_{y_2,c_2})(1)=(y_1 C_{1,k}, \dots, y_k C_{k,k})=H.\] 
Since $\tau_{dg,m}(G)=\tau_{g,m}(dG)$ (by Lemma \ref{lem:rho-tau-mults}), we multiply by $d$ and get 
\[(\tau_{dy_k,c_k}\circ \tau_{dy_{k-1},c_{k-1}}\circ\dots\circ\tau_{dy_2,c_2})(d)=(dy_1 C_{1,k}, \dots, dy_k C_{k,k})=dH.\] 
Finally, since $z_i=dy_i$ and $dH=G$,
\[(\tau_{z_k,c_k}\circ \tau_{z_{k-1},c_{k-1}}\circ\dots\circ\tau_{z_2,c_2})(d)=(z_1 C_{1,k}, \dots, z_k C_{k,k})=G.\]
\end{proof}

We now have our result. Given any two telescopic sequences $G$ and $H$ with $\gcd(G)=\gcd(H)$, we can perform a sequence of $\rho_n$ and $\tau_{g,m}$ operations on $G$ to produce $H$. We send $G$ to $\gcd(G)=\gcd(H)$ and then send $\gcd(H)$ to $H$.

\begin{prop}
Let $G$ and $H$ be telescopic sequences with $\gcd(G)=\gcd(H)$. Then there exist functions $\phi_{G,H}$ and $\phi_{H,G}$, each a finite composition of functions $\rho_{n_i}$ and $\tau_{g_j,m_j}$ (for parameters $n_i$ and $g_j, m_j$) such that $\phi_{G,H}(G)=H$ and $\phi_{H,G}(H)=G$.
\end{prop}

\begin{proof}
By Lemmas \ref{lem:rho-down} and \ref{lem:tau-up}, $T_H(R_G(G))=T_H(\gcd(G))=T_H(\gcd(H))$. Since $H$ is telescopic, $T_H(\gcd(H))=H$. Thus, $\phi_{G,H}=T_H\circ R_G$ is composition of $\rho_{n_i}$ and $\tau_{g_j,m_j}$ functions such that $\phi_{G,H}(G)=H$.

Reversing the roles of $G$ and $H$, since $G$ is telescopic, we have the same result for $\phi_{H,G}$.
\end{proof}

We conclude this section with an example.

\begin{ex}
Let $G=(4,6,9)$ and $H=(30,18,20,33)$. Then $\gcd(G)=\gcd(H)=1$, $c(G)=(2,2)$, and $c(H)=(5,3,2)$. Both sequences are telescopic. Then $R_G=\rho_2\circ\rho_3$, and
\[R_G(G)=(\rho_2\circ\rho_3)(4,6,9)=\rho_2(2,3)=1.\] For $T_H$, we first need the $z_i$ values for $H$. Since $z_i=h_i/C_{i,4}$ for $i=1,2,3,4$, we get $z_1=1$, $z_2=3$, $z_3=10$, and $z_4=33$. Therefore $T_H=\tau_{33,2}\circ\tau_{10,3}\circ\tau_{3,5}$, and
\[T_H(1)=(\tau_{33,2}\circ\tau_{10,3}\circ\tau_{3,5})(1)=(\tau_{33,2}\circ\tau_{10,3})(5,3)=\tau_{33,2}(15,9,10)=(30,18,20,33).\] Thus, for $\phi_{G,H}=T_H\circ R_G=\tau_{33,2}\circ\tau_{10,3}\circ\tau_{3,5}\circ\rho_2\circ\rho_3$, we have $\phi_{G,H}(G)=H$.

We follow the same process for the reverse direction. Here, $T_G=\tau_{9,2}\circ\tau_{3,2}$ and $R_H=\rho_2\circ\rho_3\circ\rho_4$. Therefore $\phi_{H,G}=T_G\circ R_H=\tau_{9,2}\circ\tau_{3,2}\circ\rho_2\circ\rho_3\circ\rho_4$ and $\phi_{H,G}(H)=G$.
\end{ex}

\section{Reduction of a telescopic sequence to a minimal telescopic sequence}\label{sec:reduction-telescopic-minimal}
In this section, we present a method to take a telescopic sequence $G$ and produce a minimal telescopic sequence $G'$ such that $\langle G'\rangle=\langle G\rangle$. The main result is Theorem \ref{thm:minimal-telescopic-sequence}. In the context of free numerical semigroups, this shows that every free numerical semigroup is generated by a telescopic sequence that is minimal. As we will see, the main idea is in Proposition \ref{prop:two-cases}, which says there are only two ways in which a telescopic sequence can fail to be minimal.

Before we begin, we introduce one more sequence operation and two helpful lemmas. For $G\in\mathbb{N}_0^k$ and $n=1, \dots, k$, let 
\begin{equation}
\pi_n(G)=G_{n-1}\times G_{n,k}.
\end{equation}
In other words, for $G=(g_1, \dots, g_k)$, 
\begin{equation}
\pi_n(G)=(g_1,\dots,\widehat{g_n},\dots,g_k)=(g_1, \dots, g_{n-1}, g_{n+1}, \dots, g_k)\in\mathbb{N}_0^{k-1}.
\end{equation}
The $n$th term is removed from $G$ and the indices of all subsequent terms are decreased by one.

With this notation, we have that a sequence $G\in\mathbb{N}_0^k$ is minimal if and only if $\langle \pi_n(G)\rangle\ne\langle G\rangle$ for all $n=1, \dots, k$, which occurs if and only if $g_n\not\in\langle\pi_n(G)\rangle$ for all $n=1, \dots, k$.

As in previous sections, we will continue to assume that $g_1>0$.

\begin{lemma}\label{lem:gcd-set-semigroup}
For any sequence $G\in\mathbb{N}_0^k$, let $S=\langle G\rangle$. Then $\gcd(G)=\gcd(S)$.
\end{lemma}

\begin{proof}
Let $d_G=\gcd(G)$ and $d_S=\gcd(S)$. Since $G\subset S$, $d_S\mid d_G$. For any $s\in S$, $s=\sum_{i=1}^k a_ig_i$ for some non-negative integers $a_i$. Since $d_G\mid g_i$ for all $i$, we have $d_G\mid s$ for all $s\in S$. Since $d_G$ is a common divisor for all elements of $S$, $d_G\mid d_S$. Thus $d_G=d_S$.
\end{proof}

\begin{lemma}\label{lem:basic}
Suppose $G=(g_1, \dots, g_k)\in\mathbb{N}_0^k$ is any sequence (not necessarily telescopic). If $g_n\in\langle \pi_n(G)\rangle$ for some $n$, then $\gcd(\pi_n(G))=\gcd(G)$.
\end{lemma}

\begin{proof}
Since $\langle \pi_n(G)\rangle=\langle G\rangle$, we use Lemma \ref{lem:gcd-set-semigroup} twice with $\pi_n(G)$ and $G$ to get 
\[\gcd(\pi_n(G))=\gcd(\langle \pi_n(G)\rangle)=\gcd(\langle G\rangle)=\gcd(G).\]
\end{proof}

\subsection{Determination of the two cases for non-minimality of a telescopic sequence}
As we will see, there are two ways for which a telescopic sequence can fail to be minimal. In Lemmas \ref{lem:c_k-divides_a_k} and \ref{lem:cj-divides-aj}, we will show that if one entry of a telescopic sequence can be written as a non-negative linear combination of the remaining elements, then we have a divisibility condition on the coefficient of the term with highest index.

\begin{lemma}\label{lem:c_k-divides_a_k}
Let $G=(g_1, \dots, g_k)\in\mathbb{N}_0^k$ be telescopic 
with $c(G)=(c_2, \dots, c_k)$. Suppose $g_n\in\langle\pi_n(G)\rangle$ for some $n\in\{1, \dots, k\}$. That is, suppose 
\begin{equation}\label{eqn:gn-summation}
g_n = \sum_{i=1,\,i\ne n}^k a_i g_i
\end{equation}
for non-negative integers $a_i$. Then $c_k \mid a_k$.
\end{lemma}

\begin{proof}
Since $g_n\in\langle \pi_n(G)\rangle$, by Lemma \ref{lem:basic}, $\gcd(\pi_n(G))=\gcd(G)$. We consider two cases: $n=k$ and $n<k$.

Assume $n=k$. Since $G_{k-1}=(g_1, \dots, g_{k-1})=\pi_k(G)$, we have $\gcd(G_{k-1})=\gcd(G_k)$, so $c_k=\gcd(G_{k-1})/\gcd(G_k)=1$, and consequently $c_k\mid a_k$.

Now assume $n<k$. Since $\gcd(G_k)=\gcd(G)$, then $c_k=\gcd(G_{k-1})/\gcd(G)$. Therefore, in order to have that $c_k\mid a_k$, it is enough to show that $\gcd(G_{k-1})\mid (a_k\gcd(G))$. From Equation \eqref{eqn:gn-summation} we solve for $a_kg_k$ to get 
\[a_kg_k = g_n-\sum\limits_{i=1,\,i\ne n}^{k-1}a_ig_i,\]
so $a_kg_k$ is a linear combination of $g_1, \dots, g_{k-1}$. Since $\gcd(G_{k-1}) \mid g_1, \dots, g_{k-1}$, we have $\gcd(G_{k-1}) \mid (a_kg_k)$. Then, since $\gcd\left(\gcd(G_{k-1}), g_k\right)=\gcd(G)$, we conclude that $\gcd(G_{k-1})\mid (a_k\gcd(G))$.
\end{proof}

\begin{lemma}\label{lem:cj-divides-aj}
Let $G=(g_1, \dots, g_k)\in\mathbb{N}_0^k$ be a telescopic sequence with $c(G)=(c_2, \dots, c_k)$. Suppose there exist $n$ and $m$ with $1\le n < m\le k$ such that $g_n\in\langle \pi_n(G_m)\rangle$. That is, suppose 
\begin{equation}
g_n = \sum_{i=1,\,i\ne n}^m a_i g_i
\end{equation}
for non-negative integers $a_i$.  Then $c_m\mid a_m$.
\end{lemma}

\begin{proof}
Suppose $g_n\in\langle\pi_n(G_m)\rangle$ for some $m>n$. Let $H=G_m$, a sequence of length $m$. By Lemma \ref{lem:shortened-telescopic}, $H$ is telescopic with $c(H)=(c_2,\dots,c_m)$. Since $g_n\in\langle \pi_n(H)\rangle$, we apply Lemma \ref{lem:c_k-divides_a_k} to $H$ to conclude that $c_m\mid a_m$.
\end{proof}

We need one more lemma before we can state our main result on the form of a non-minimal telescopic sequence.

\begin{lemma}\label{lem:telescopic-exactly-one-occurs}
For $G=(g_1, \dots, g_k)\in\mathbb{N}_0^k$ telescopic with $c(G)=(c_2, \dots, c_k)$, suppose $g_n\in \langle \pi_n(G_m) \rangle$ for some $m$ and $n$ with $1\le n\le m\le k$. Then $g_n\in \langle \pi_n(G_{m-1})\rangle$ or $g_n=c_m g_m$.
\end{lemma}

\begin{proof}
First, suppose $n=m$.  Since $\pi_n(G_n)=(g_1, \dots, g_{n-1})=\pi_n(G_{n-1})$, if $g_n\in\langle \pi_n(G_m)\rangle$ then $g_n\in\langle\pi_n(G_{m-1})\rangle$, as desired. (Additionally, we have that $c_m=\gcd(G_m)/\gcd(G_{m-1})=1$ in this case. Since $g_n=g_m$, we also have $g_n=c_mg_m$.)

Now suppose $n<m$. Suppose $g_n\in\langle\pi_n(G_m)\rangle$ and $g_n\not\in\langle \pi_n(G_{m-1})\rangle$. We wish to show that $g_n=c_mg_m$.

Since $g_n\in\langle\pi_n(G_m)\rangle$, there exist non-negative integers $a_i$ such that 
\begin{equation}\label{eqn:a_i-eqn}
g_n=\sum_{i=1,\, i\ne n}^m a_ig_i.
\end{equation} 
Since $g_n\not\in\langle\pi_n(G_{m-1})\rangle$, we must have $a_m > 0$. By Lemma \ref{lem:cj-divides-aj}, $c_m\mid a_m$, so we let $q=a_m/c_m\in\mathbb{N}$. For $G$ telescopic, $c_mg_m\in\langle G_{m-1}\rangle$, so there are $b_1, \dots, b_{m-1}\in\mathbb{N}_0$ such that 
\begin{equation}\label{eqn:telescopic-eqn}
c_mg_m=b_1g_1+\cdots+b_{m-1}g_{m-1}.
\end{equation} 
Thus $a_m g_m = qc_mg_m = q\sum\limits_{i=1}^{m-1}b_ig_i$, which we plug into Equation \eqref{eqn:a_i-eqn} to obtain 
\begin{equation}\label{eqn:combined-eqn}
g_n = qb_ng_n + \sum\limits_{i=1,\,i\ne n}^{m-1}(qb_i+a_i)g_i.\end{equation}
Since $g_n\not\in\langle \pi_n(G_{m-1})\rangle$, we must have $qb_n\ne0$. Since $qb_n\in\mathbb{N}$, this implies $qb_n=1$, so $(qb_i+a_i)g_i=0$ for $i=1, \dots, m-1$, with $i\ne n$. Therefore $q=b_n=1$ and $a_ig_i=b_ig_i=0$ for $i=1, \dots, m-1$, with $i\ne n$. Since $q=1$, $a_m=c_m$. Finally, we plug these values into either Equation \eqref{eqn:a_i-eqn} or Equation \eqref{eqn:telescopic-eqn} to conclude that $g_n=c_mg_m$.
\end{proof}

We now have our result.

\begin{prop}\label{prop:two-cases}
For $G=(g_1, \dots, g_k)\in\mathbb{N}_0^k$ a telescopic sequence, suppose $g_n\in\langle \pi_n(G)\rangle$. Then $g_n\in\langle G_{n-1} \rangle$ or $g_n=c_m g_m$ for some $m>n$.
\end{prop}
\begin{proof}
This follows immediately from Lemma \ref{lem:telescopic-exactly-one-occurs}.
\end{proof}
Therefore, if a telescopic sequence $G$ is not minimal, there are two cases: Case 1, where $g_n\in\langle G_{n-1}\rangle$; and Case 2, where $g_n=c_mg_m$ for some $m>n$.

Proposition \ref{prop:two-cases} gives us a straightforward criterion to determine whether a given telescopic sequence is minimal.
\begin{cor}\label{cor:minimality-criterion}
Suppose $G=(g_1,\dots,g_k)\in\mathbb{N}_0^k$ is telescopic with $c(G)=(c_2,\dots,c_k)$. Let $H=(h_2,\dots,h_k)=(c_2g_2,\dots,c_kg_k)$. Then $G$ is minimal if and only if $g_i\ne h_j$ for all $i,j$.
\end{cor}
\begin{proof}
$(\implies)$ Suppose $G$ is minimal. Then $g_j\not\in\langle G_{j-1}\rangle$ for all $j\ge2$. Since $G$ is telescopic, $c_jg_j\in\langle G_{j-1}\rangle$ for all $j\ge2$. Thus $g_j\ne c_jg_j$ for all $j\ge2$.

Also, since $G$ is minimal, $g_i\not\in\langle g_j\rangle$ for all $i\ne j$. In particular, we have $g_i\ne c_jg_j$ for all $i\ne j$. Combined with the previous paragraph, we have $g_i\ne h_j$ for all $i,j$.

$(\impliedby)$ Suppose $G$ is not minimal, so $g_i\in\langle\pi_i(G)\rangle$ for some $i$. By Proposition \ref{prop:two-cases}, we have either $g_i\in\langle \pi_i(G_{i-1})\rangle$ or $g_i=c_jg_j$ for some $j>i$. In the first case, $\gcd(g_1,\dots,g_i)=\gcd(g_1,\dots,g_{i-1})$, so $c_i=1$. Therefore $g_i=c_ig_i=h_i$. In the second case, $g_i=h_j$.
\end{proof}

The following example illustrates the two cases of Proposition \ref{prop:two-cases}.

\begin{ex}\label{ex:removal-example}
Let $G=(660,550,352,902,50,201)$ a telescopic sequence with $c(G)=(6,5,1,11,2)$. We have $g_2\in\langle \pi_2(G)\rangle$ and $g_4\in\langle \pi_4(G)\rangle$. Observe that $g_2=c_5g_5$, which is Case 2. And $g_4=g_2+g_3$, which is Case 1.
\end{ex}
For each case, we will see how to remove $g_n$ without losing the telescopic property of the generating sequence. We will illustrate our methods with the above example.

\subsubsection{Case 1}
For the case where $g_n\in\langle G_{n-1}\rangle$, we use a result (Proposition \ref{prop:telescopic-operations-rho}) about $\rho_n$ from Section \ref{sec:sequence-operations}.

\begin{cor}\label{cor:remove_gn_telescopic}
Suppose $G=(g_1, \dots, g_k)\in\mathbb{N}_0^k$ is telescopic with $g_1>0$ and that $g_n\in\langle G_{n-1} \rangle$ for some $n$. Then $n>1$, $c_n=1$, and $\pi_n(G)$ is a telescopic sequence with $\langle \pi_n(G) \rangle = \langle G \rangle$.
\end{cor}

\begin{proof}
Suppose $n=1$. Then $G_{n-1}=G_0=()$, the empty sequence, so $\langle G_0\rangle=\{0\}$. Thus $g_1=0$. However, since we have assumed that $g_1>0$, this cannot occur. Therefore $n>1$. 

Now, let $c(G)=(c_2, \dots, c_k)$. If $g_n\in\langle G_{n-1}\rangle$, then $\langle G_{n-1}\rangle=\langle G_n\rangle$ so 
\[c_n=\gcd(G_{n-1})/\gcd(G_n)=1.\] 
By Proposition \ref{prop:telescopic-operations-rho}, the sequence $\rho_n(G)$ is telescopic. Since $c_n=1$,  
\[\rho_n(G) = \left(g_1/c_n, \dots, g_{n-1}/c_n, g_{n+1}, \dots, g_k\right)  =(g_1, \dots, g_{n-1}, g_{n+1}, \dots, g_k) = \pi_n(G),\] so $\pi_n(G)$ is telescopic. Also, since $\langle G_{n-1}\rangle=\langle G_n\rangle$, $\langle G_{n-1}\times G_{n,k}\rangle = \langle G_n\times G_{n,k}\rangle$, so $\langle \pi_n(G)\rangle=\langle G\rangle$, as desired.
\end{proof}

In this case, we can therefore remove $g_n$ to produce a shorter telescopic sequence which generates the same submonoid.

\begin{ex}\label{ex:case1}
For the telescopic sequence $G=(660,550,352,902,50,201)$ with $c(G)=(6,5,1,11,2)$ from Example \ref{ex:removal-example}, $g_4=902=550+352=g_2+g_3\in\langle G_3\rangle$. To remove $g_4$ from $G$, we compute $\pi_4(G)=(660,550,352,50,201)$, a telescopic sequence with $c(\pi_4(G))=(6,5,11,2)$ and $\langle \pi_4(G)\rangle=\langle G\rangle$.
\end{ex}

\subsubsection{Case 2}
Now we consider the case where $g_n=c_m g_m$ for $m>n$. Instead of just removing $g_n$ (like we did in Case 1), we will first swap the positions of $g_n$ and $g_m$ before removing $g_n$. As the following lemma states, this is equivalent to a permutation from $S_{k-1}$ acting on $\pi_n(G)$.

\begin{lemma}\label{lem:permutation} 
For $G\in\mathbb{N}_0^k$, suppose $1\le n<m\le k$ and consider the transposition $(n\; m)\in S_k$. Then $\pi_m ( (n\;m) (G)) = \sigma(\pi_n (G))$ for $\sigma\in S_{k-1}$ defined by 
\begin{equation}
\sigma(i)=
\begin{cases}
i, & \text{if $i<n$ or $i\ge m$;} \\ 
n, & \text{if $i=m-1$;} \\ 
i+1, & \text{if $n\le i< m-1$.} 
\end{cases} \end{equation}
\end{lemma}

We now have a lemma with two results about greatest common divisors which will be useful for the main result of the $g_n=c_mg_m$ case. Since $c_1$ is not defined, we will assume $n>1$ for the following lemma and proposition. We will then address the possibility of $n=1$ in Theorem \ref{thm:minimal-telescopic-sequence} with the help of the permutation $(1\;2)$.

\begin{lemma}\label{lem:gcd(qn,cj)=1}
For $G=(g_1, \dots, g_k)\in\mathbb{N}_0^k$ a telescopic sequence with $c(G)=(c_2, \dots, c_k)$ and $\gcd(G)=d$, if $g_n=c_m g_m$ for some $m,n$ with $1<n<m\le k$, 
then $g_n=z_n C_{n,k}$ with $z_n\in\langle d\rangle$ and $\gcd(z_n/d, c_n)=1$. In addition we have $\gcd(z_n/d, c_m)=1$ and $\gcd(c_j, c_m)=1$ for all $n<j<m$.
\end{lemma}

\begin{proof}
By Proposition \ref{prop:explicit-description-telescopic}, $g_n=z_n C_{n,k}$ with $\gcd(z_n/d, c_n)=1$ and $g_m=z_m C_{m,k}$ with $\gcd(z_m/d, c_m)=1$. Since $g_n=c_m g_m$, we have $z_n C_{n,k} = c_m z_m C_{m,k}$, so $z_n C_{n,m-1} = z_m$. Therefore $(z_n/d) C_{n,m-1} = (z_m/d)$.

For the first greatest common divisor statement, $(z_n/d)\mid (z_m/d)$, so $\gcd(z_n/d, c_m)\mid \gcd(z_m/d, c_m)$, which is 1. Thus $\gcd(z_n/d, c_m)=1$.

For the second greatest common divisor statement, for $n<j<m$ since $c_j\mid C_{n,m-1}$, $c_j\mid (z_m/d)$. Then $\gcd(c_j, c_m)\mid \gcd(z_m/d, c_m)$, which is 1. Thus $\gcd(c_j, c_m)=1$.
\end{proof}

We now have our main result for this case.

\begin{prop}\label{prop:replace-the-multiple}
For $G=(g_1, \dots, g_k)\in\mathbb{N}_0^k$ telescopic with $c(G)=(c_2, \dots, c_k)$, suppose $g_n=c_m g_m$ for some $m,n$ with $1<n<m\le k$.
For the transposition $(n\; m)\in S_k$, the sequence $\pi_m((n\; m)(G))$ is telescopic and $\langle \pi_m((n\; m)(G))\rangle = \langle G\rangle$.
\end{prop}

\begin{proof}
For $H=\pi_m((n\; m)(G))$, by Lemma \ref{lem:permutation}, $H$ is a permutation of $\pi_n(G)$, and since $g_n\in\langle \pi_n(G)\rangle$, 
$\langle H\rangle=\langle\pi_n(G)\rangle=\langle G\rangle$.

We have $H=(h_1, \dots, h_{k-1})\in\mathbb{N}_0^{k-1}$ where 
\[h_i = 
\begin{cases}
g_i, & \text{for $i=1, \dots, m-1$ with $i\ne n$;} \\
g_m, & \text{for $i=n$;} \\
g_{i+1}, & \text{for $i=m, \dots, k-1$.}
\end{cases}
\]
Let $c(H)=(e_2, \dots, e_{k-1})$. 
To show $H$ is telescopic, we must show $e_j h_j\in \langle H_{j-1}\rangle$ for $2\le j\le k-1$.
We will consider five subintervals of indices: $2\le j < n$; $j=n$; $n<j<m$; $j=m$; and $m< j\le k-1$.

Since $H_{n-1}=G_{n-1}$ and $G$ is telescopic, by Lemma \ref{lem:shortened-telescopic} $H_{n-1}$ is telescopic, which implies that $e_jh_j\in\langle H_{j-1}\rangle$ for $2\le j\le n-1$.

Suppose $j=n$. Let $d=\gcd(G)$. By Corollary \ref{cor:explicit-description-telescopic}, $g_n=z_nC_{n,k}$ for some $z_n\in d\mathbb{N}_0$ with $\gcd(z_n/d, c_n)=1$. Then \begin{align*}
\gcd(H_n)
&=\gcd(h_1, \dots, h_n)\\
&=\gcd(g_1, \dots, g_{n-1}, g_m)\\
&=\gcd(\gcd(G_{n-1}), g_n/c_m)\\
&=\gcd(C_{n-1,k}d, z_nC_{n,k}/c_m) \\
&=\frac{C_{n,k}d}{c_m}\gcd(c_m c_n, z_n/d).
\end{align*}
Note that $C_{n,k}/c_m\in\mathbb{N}$ since $m>n$. Since $\gcd(z_n/d, c_n)=1$ and $\gcd(z_n/d, c_m)=1$ (the latter from Lemma \ref{lem:gcd(qn,cj)=1}), $\gcd(H_n)=C_{n,k}d/c_m$. Since $\gcd(H_{n-1})=\gcd(G_{n-1})$, we have $e_n=\gcd(H_{n-1})/\gcd(H_n)=\gcd(G_{n-1})/(C_{n,k}d/c_m)=c_mc_n$, so $e_nh_n=c_mc_n g_n/c_m = c_ng_n$ which is in $\langle G_{n-1}\rangle$ since $G$ is telescopic. And since $H_{n-1}=G_{n-1}$, we have $e_nh_n\in\langle H_{n-1}\rangle$.

From the $j=n$ case, we saw $\gcd(H_n)=C_{n,k}d/c_m$. By Corollary \ref{cor:explicit-description-telescopic} again, $g_{n+1}=z_{n+1}C_{n+1,k}$ for some $z_{n+1}\in d\mathbb{N}$ with $\gcd(z_{n+1}/d,c_{n+1})=1$. Then, 
\begin{align*}
\gcd(H_{n+1})
&=\gcd(g_1, \dots, g_{n-1}, g_m, g_{n+1})\\
&=\gcd(C_{n,k}d/c_m, g_{n+1})\\
&= \gcd(C_{n,k}d/c_m, z_{n+1}C_{n+1, k})\\
&= \frac{C_{n+1,k}d}{c_m} \gcd(c_{n+1}, c_m z_{n+1}/d)\\
&= \frac{C_{n+1,k}d}{c_m},
\end{align*}
since $\gcd(c_{n+1},z_{n+1}/d)=1$ and $\gcd(c_{n+1}, c_m)=1$ (the latter from Lemma \ref{lem:gcd(qn,cj)=1}).

Continuing in this way, we find that $\gcd(H_i)=C_{i,k}d/c_m$ for $n\le i<m$. Therefore, for $n<j<m$, $e_j=\gcd(H_{j-1})/\gcd(H_j)=c_j$. Since $G$ is telescopic, $c_jg_j\in\langle G_{j-1}\rangle$. With $g_n=c_m g_m$ and $h_n=g_m$, we have $g_n\in\langle h_n\rangle$, so $\langle G_{j-1}\rangle\subset\langle H_{j-1}\rangle$. Thus $e_jh_j=c_jg_j\in\langle G_{j-1}\rangle\subseteq\langle H_{j-1}\rangle$ for $n<j<m$.

Suppose $j=m$. We first note that $\langle G_m\rangle=\langle H_{m-1}\rangle$. Then, since $H_m$ is a permutation of $\pi_n(G_{m+1})$, we have $\langle H_m\rangle=\langle \pi_n(G_{m+1})\rangle=\langle G_{m+1}\rangle$, the latter equality holding because $g_n\in\langle g_m\rangle\subset\langle \pi_n(G_{m+1})\rangle$. Thus $\gcd(H_m)=\gcd(G_{m+1})=C_{m+1,k}d$, from which it follows that $e_m=\gcd(H_{m-1})/\gcd(H_m)=(C_{m-1,k}d/c_m)/(C_{m+1,k}d)=c_{m+1}$. Since $G$ is telescopic, $c_{m+1}g_{m+1}\in\langle G_m\rangle$. Then $e_mh_m=c_{m+1}g_{m+1}\in\langle G_m\rangle=\langle H_{m-1}\rangle$.

Finally, for $m\le i\le k-1$, $H_i$ is a permutation of $\pi_n(G_{i+1})$, so $\langle H_i\rangle=\langle G_{i+1}\rangle$ and $\gcd(H_i)=\gcd(G_{i+1})=C_{i+1,k}d$. Then for $m<j\le k-1$, $e_j=\gcd(H_{j-1})/\gcd(H_j)=\gcd(G_{j})/\gcd(G_{j+1})=c_{j+1}$. Since $G$ is telescopic, $c_{j+1}g_{j+1}\in \langle G_j\rangle$. Thus $e_jh_j=c_{j+1}g_{j+1}\in\langle G_{j}\rangle=\langle H_{j-1}\rangle$ for $m< j\le k-1$.
\end{proof}

As a result, we can therefore remove $g_n$ and reorder the remaining terms to produce a shorter telescopic sequence which generates the same submonoid.

\begin{ex}\label{ex:case2}
For the telescopic sequence $G=(660,550,352,902,50,201)$ with $c(G)=(6,5,1,11,2)$ from Example \ref{ex:removal-example}, $g_2=550=11\cdot50=c_5g_5$. To remove $g_2$ from $G$, we compute \[\pi_5((2\;5)(G))=\pi_5(660,50,352,902,550,201)=(660,50,352,902,201).\] Then $c(\pi_5((2\;5)(G)))=(66,5,1,2)$, $\pi_5((2\;5)(G))$ is telescopic, and $\langle \pi_5((2\;5)(G))\rangle=\langle G\rangle$.
\end{ex}

\subsection{Answers to Questions \ref{q:1} and \ref{q:2}}
If a telescopic sequence $G$ is not minimal, then there is some $n$ such that $g_n\in\langle\pi_n(G)\rangle$. In Cases 1 and 2, we saw how to remove $g_n$ and produce a shorter telescopic sequence which generates the same submonoid. We can now answer Question \ref{q:1}.

\begin{thm}\label{thm:minimal-telescopic-sequence}
Let $G\in\mathbb{N}_0^k$ be telescopic. Then there exists a minimal telescopic sequence $G'$ such that $\langle G'\rangle=\langle G\rangle$.
\end{thm}

\begin{proof}
If $G$ is minimal, we are done. Otherwise, for $G=(g_1, \dots, g_k)$, there is some $n$ for which $g_n\in\langle \pi_n(G) \rangle$. By Proposition \ref{prop:two-cases}, either $g_n\in\langle G_{n-1}\rangle$ (Case 1) or, for $c(G)=(c_2, \dots, c_k)$, $g_n=c_m g_m$ for some $m>n$ (Case 2).

For Case 1, if $g_n\in\langle G_{n-1}\rangle$, then by Corollary \ref{cor:remove_gn_telescopic}, $\pi_n(G)$ is a telescopic sequence of length $k-1$ such that $\langle \pi_n(G)\rangle = \langle G\rangle$.

For Case 2, if $g_n=c_m g_m$ for some $m>n$, we consider two cases of $n$. If $n>1$, then by Proposition \ref{prop:replace-the-multiple}, $\pi_m((n\; m)(G))$ is a telescopic sequence of length $k-1$ such that $\langle \pi_m((n\; m)(G))\rangle = \langle G\rangle$. 

If $n=1$ and $g_2>0$, we transpose the first two terms and use the previous paragraph to remove the (new) second term. Let $H=(1\;2)(G)$. We have $H=(h_1,\dots,h_k)$, a telescopic sequence (by Proposition \ref{prop:swap-first-two}) with $h_1>0$, $h_2\in\langle \pi_2(H)\rangle$, and $\langle H\rangle=\langle G\rangle$. By Proposition \ref{prop:replace-the-multiple}, $\pi_m((2\; m)(H))$ is a telescopic sequence of length $k-1$ such that $\langle \pi_m((2\; m)(H))\rangle=\langle H\rangle=\langle G\rangle$.

Now, to construct the minimal telescopic sequence, starting with $G$ we first iteratively remove all of the Case 1 terms (in any order). Observe that this removes any zeros. Then, from the resulting sequence, we can iteratively remove all of the Case 2 terms (with $n>1$ and $n=1$). Since $G\in\mathbb{N}_0^k$ and we remove one term which each step, this process terminates in fewer than $k$ steps, producing a telescopic sequence $G'$ that is necessarily minimal and satisfies $\langle G'\rangle=\langle G\rangle$.
\end{proof}

\begin{ex}
Once again using the telescopic sequence $G=(660,550,352,902,50,201)$ with $c(G)=(6,5,1,11,2)$ from Example \ref{ex:removal-example}, we have $g_n\in\langle\pi_n(G)\rangle$ for $n=2$ and $n=4$. In Example \ref{ex:case1}, we removed $g_4$ from $G$ to get the telescopic sequence $H=\pi_4(G)=(660,550,352,50,201)$ with $c(H)=(6,5,11,2)$. In Example \ref{ex:case2}, we removed $g_2$ from $G$ to get the telescopic sequence $\bar{H}=\pi_5((2\;5)(G))=(660,50,352,902,201)$ with $c(\bar{H})=(66,5,1,2)$. Of course, neither of these resulting sequences is minimal.

In $H=(h_1,\dots,h_5)$, $h_2=11 h_4$, which is Case 2. We therefore compute $\pi_4((2\;4)(H))=\pi_4(660,50,352,550,201)=(660,50,352,201)$, a telescopic sequence that generates the same submonoid as $H$. Observe that this sequence is also minimal. (On its own, this paragraph addresses the motivating example in Example \ref{ex:question-example}.)

In $\bar{H}=(\bar{h}_1,\dots,\bar{h}_5)$, $\bar{h}_4=11\bar{h}_2+1\bar{h}_3$, which is Case 1. We therefore compute $\pi_4(\bar{H})=(660,50,352,201)$, which is the same sequence as in the previous paragraph.

The result is that, starting with the telescopic sequence $G=(660,550,352,902,50,201)$, after two steps we produce the sequence $G'=(660,50,352,201)$ for which $c(G')=(66,5,2)$. Then $G'$ is telescopic and minimal, and $\langle G'\rangle=\langle G\rangle$.
\end{ex}

Since a free numerical semigroup is one which is generated by a (not necessarily minimal) telescopic sequence $G$ with $\gcd(G)=1$, we get the following.
\begin{cor}\label{cor:free-gen-by-telescopic}
Suppose $S$ is a free numerical semigroup. Then $S$ is generated by a telescopic sequence which is minimal.
\end{cor}

We conclude this section by answering Question \ref{q:2}.
\begin{cor}
Let $G$ and $H$ be sequences with $\langle G\rangle=\langle H\rangle$. Then a permutation of $G$ is telescopic if and only if a permutation of $H$ is telescopic.
\end{cor}

\begin{proof}
Suppose $G$ is telescopic and that $\langle G\rangle=\langle H\rangle$. By Theorem \ref{thm:minimal-telescopic-sequence} there is a minimal telescopic sequence $G'$ such that $\langle G'\rangle=\langle G\rangle$, so $\langle G'\rangle=\langle H\rangle$ as well. Since $G'$ is minimal, each term in $G'$ must appear in $H$. Consider the permutation \[\sigma(H)=G' \times H',\] where $H'$ consists of the remaining terms of $H$ not included in $G'$. Each term in $H'$ is in $\langle G'\rangle$, so $c(\sigma(H))=c(G')\times(1, \dots, 1)$. Since $G'$ is telescopic, the telescopic conditions hold for the $G'$ portion of $\sigma(H)$. And since 1 times each term in $H'$ is in $\langle G'\rangle$, the telescopic conditions hold for the $H'$ portion of $\sigma(H)$. Therefore, $\sigma(H)$ is a telescopic sequence.

Switching the roles of $G$ and $H$, we find the reverse implication holds as well.
\end{proof}

\section{Construction of minimal telescopic sequences with desired properties}\label{sec:explicit-minimal-telescopic-construction}
We now look at necessary and sufficient conditions for a constructed telescopic sequence to be minimal.  Suppose $G=(g_1, \dots, g_k)\in\mathbb{N}_0^k$ and $c(G)=(c_2, \dots, c_k)$. If $G$ is telescopic and not minimal, then $g_n\in\langle \pi_n(G)\rangle$ for some $n$, so by Proposition \ref{prop:two-cases}, either 
\begin{enumerate}
\item $g_n\in \langle G_{n-1}\rangle$; or
\item there is some $m>n$ such that $g_m \mid g_n$.
\end{enumerate}
The first case occurs precisely when $c_n=1$, which we can avoid by requiring $c_j>1$ for all $j=2, \dots, k$. The second case occurs precisely when $z_mC_{m,k}\mid z_nC_{n,k}$, which we can avoid by requiring $z_j\nmid z_i C_{i,j}$ for all $i$, $j$ with $1\le i<j\le k$. 

\begin{cor}\label{cor:explicit-min}
Suppose $G=(g_1, \dots, g_k)$ is a telescopic sequence (with notation as in Corollary \ref{cor:explicit-description-telescopic}). Then $G$ is minimal if and only if we additionally have $c_j>1$ for $j=2, \dots, k$ and $z_j \nmid z_i C_{i,j}$ for all $1\le i<j\le k$.
\end{cor}

\begin{ex}
Suppose we want a free numerical semigroup $S=\langle G\rangle$ where $G$ is minimal and telescopic with \[c(G)=(c_2, c_3, c_4, c_5)=(2, 3, 4, 5).\]

To generate a numerical semigroup, we need $\gcd(G)=1$, so we have $z_1=\gcd(G)=1$. Since $c_j>1$ for all $j$, for minimality we need only satisfy the non-divisibility conditions for $z_2,\dots,z_5$. We need $z_j\in\mathbb{N}$ where 
\begin{itemize}
\item $z_2\in\langle z_1\rangle$, $\gcd(z_2, 2)=1$, and $z_2\nmid 2z_1$;
\item $z_3\in\langle 2z_1, z_2\rangle$, $\gcd(z_3, 3)=1$, and $z_3\nmid 6z_1, 3z_2$;
\item $z_4\in\langle 6z_1, 3z_2, z_3\rangle$, $\gcd(z_4, 4)=1$, and $z_4\nmid 24z_1, 12z_2, 4z_3$;
\item $z_5\in\langle 24z_1, 12z_2, 4z_3,  z_4\rangle$, $\gcd(z_5, 5)=1$, and $z_5\nmid 120z_1, 60z_2, 20z_3, 5z_4$.
\end{itemize}

Any such $z_2,\dots,z_5$ satisfying the above will produce the desired result, and conversely every such numerical semigroup $S$ can be constructed in this way.

For a concrete example, we can take $z_2=3$, $z_3=5$, $z_4=11$, and $z_5=22$. For $g_i=z_i C_{i,5}$, we get $G=(120, 180, 100, 55, 22)$. We check that $\gcd(G)=1$, $G$ is minimal, and $G$ is telescopic with $c(G)=(2, 3, 4, 5)$, as desired.
\end{ex}

\begin{cor}\label{cor:increasing-telescopic-minimal}
Suppose $G=(g_1, \dots, g_k)\in\mathbb{N}^k$ is a telescopic and non-decreasing sequence with $c(G)=(c_2, \dots, c_k)$ where $c_j>1$ for all $j$. Then $G$ is minimal.

In particular, if $\gcd(G)=1$, then for the numerical semigroup $S=\langle G\rangle$, we have $e(S)=|G|$.
\end{cor}

\begin{proof}
If $g_j=g_{j-1}$ for some $j>1$, then $c_j=1$. Since we assume $c_j>1$ for all $j$, we must have $g_j\ne g_{j+1}$, so $G$ is a strictly increasing sequence. In other words, $g_n<g_m$ whenever $n<m$.

Since $G$ is telescopic and $c_j>1$ for all $j$, $G$ is minimal if and only if $g_m\nmid g_n$ for all $1\le n<m\le k$. Since $g_m>g_n>0$, we have $g_m\nmid g_n$. Therefore $G$ is minimal.
\end{proof}

\section{Acknowledgments}
The author wishes to thank the anonymous referee for substantially helping to improve the presentation of this paper in many ways. Of particular note is the suggestion to write Lemma \ref{lem:shortened-telescopic}, which simplified a few subsequent proofs.

\bibliographystyle{jis}
\bibliography{refs}

\bigskip
\hrule
\bigskip

\noindent 2010 {\it Mathematics Subject Classification}: Primary 20M14;
Secondary 11B75.

\noindent \emph{Keywords:}
telescopic sequence, free numerical semigroup, minimal generating sequence, additive submonoid.

\end{document}